\pdfoutput=1
\RequirePackage{ifpdf}
\ifpdf 
\documentclass[pdftex]{sigma}
\else
\documentclass{sigma}
\fi

\newtheorem{Theorem}{Theorem}[section]
\newtheorem*{Theorem*}{Theorem}
\newtheorem{Corollary}[Theorem]{Corollary}
\newtheorem{Lemma}[Theorem]{Lemma}

 { \theoremstyle{definition}

\newtheorem{Example}[Theorem]{Example}
\newtheorem{Remark}[Theorem]{Remark} }

\numberwithin{equation}{section}

\begin{document}
\allowdisplaybreaks

\newcommand{\arXivNumber}{2205.08153}

\renewcommand{\PaperNumber}{069}

\FirstPageHeading

\ShortArticleName{Freezing Limits for Beta-Cauchy Ensembles}

\ArticleName{Freezing Limits for Beta-Cauchy Ensembles}

\Author{Michael VOIT}

\AuthorNameForHeading{M.~Voit}
\Address{Fakult\"at Mathematik, Technische Universit\"at Dortmund,\\ Vogelpothsweg 87, D-44221 Dortmund, Germany}
\Email{\href{mailto:michael.voit@math.tu-dortmund.de}{michael.voit@math.tu-dortmund.de}}
\URLaddress{\url{http://www.mathematik.tu-dortmund.de/lsiv/voit/voit.html}}

\ArticleDates{Received May 19, 2022, in final form September 15, 2022; Published online September 28, 2022}

\Abstract{Bessel processes associated with the root systems $A_{N-1}$ and $B_N$ describe interacting particle systems with $N$ particles on $\mathbb R$; they form dynamic versions of the classical $\beta$-Hermite and Laguerre ensembles. In this paper we study corresponding Cauchy processes constructed via some subordination. This leads to $\beta$-Cauchy ensembles in both cases with explicit distributions. For these distributions we derive central limit theorems for fixed $N$ in the freezing regime, i.e., when the parameters tend to infinity. The results are closely related to corresponding known freezing results for $\beta$-Hermite and Laguerre ensembles and for Bessel processes.}


\Classification{60F05; 60B20; 70F10; 82C22; 33C45}

\section{Introduction}

By a classical result in probability (see, e.g., \cite{BF, Sa}), a subordination of Brownian motions on~$\mathbb R^N$
by inverse Gaussian L\'evy processes
on $[0,\infty[$ leads to classical Cauchy processes on~$\mathbb R^N$. In the one-dimensional case and for a start in the origin,
 these Cauchy processses $(X_t)_{t\ge0}$ are Cauchy-distributed with the densities
\[
f_t(x)=\frac {1}{\pi } {\frac {t}{t^{2}+x^{2}}},\qquad x\in\mathbb R, \quad t>0.
\]

Motivated by the theory of Bessel processes
associated with root systems on Weyl chambers $C_N\subset\mathbb R^N$
and the distributions of the classical $\beta$-Hermite and Laguerre ensembles, one can
transfer this
subordination principle from Brownian motions to these Bessel processes
and obtain some kinds of Cauchy--Bessel processes on $C_N$.
This construction in particular leads to Lebesgue densities of the form
\begin{equation}\label{density-general-cauchy-intro}
C(k,N)\cdot
 \frac{1}{( 1+\|y\|^2)^{\gamma_k+(N+1)/2}} w_k(y)
\end{equation}
with some constant $\gamma_k\ge0$, a norming constant $C(k,N)>0$, and some weight functions $w_k$ where $k$
is some positive, possibly multivariate multiplicity constant; see \cite{RV1}. For the most
relevant root systems of types $A_{N-1}$ and $B_N$, these
 weights are given by
\[
w_k^A(x):= \prod_{i,j\colon i<j}(x_i-x_j)^{2k}, \qquad
w_k^B(x):= \prod_{i,j\colon i<j}\big(x_i^2-x_j^2\big)^{2k_2} \prod_{i=1}^N x_i^{2k_1}
\]
with $k,k_1,k_2\ge0$ respectively.
Due to the analogous construction and shape to the classical setting, we call the distributions with the densities
(\ref{density-general-cauchy-intro}) Cauchy--Bessel distributions of types A or B respectively.

We shall prove explicit central limit theorems (CLTs) for these distributions for fixed dimensions $N$
when the parameters $k$ or $(k_1,k_2)$ tend to infinity. The limit distributions here are non-Gaussian and live on certain
halfspaces in $\mathbb R^N$ where the limit distributions are composed in some way of a $(N-1)$-dimensional normal distribution and
some distribution on $[0,\infty[$ which is related to inverse Gaussian distributions.
For the details for the types A or B we refer to Theorems \ref{clt-main-cauchy-a2} and \ref{clt-main-cauchy-b} below.
We point out that the identification of the $(N-1)$-dimensional subspaces as well as of the covariance matrices of the $(N-1)$-dimensional normal
distributions are expressed in terms of the ordered zeroes of
the classical Hermite polynomal~$H_N$ and some Laguerre polynomial~$L_N^{(\alpha)}$ respectively.
We shall present two different proofs for the central limit Theorems~\ref{clt-main-cauchy-a2} and~\ref{clt-main-cauchy-b}, where both are closely related
to the corresponding CLTs for the Bessel processes of types A and B as well for $\beta$-Hermite and Laguerre ensembles
in \cite{AHV, AKM1, AKM2, AV2, AV1,DE2,F, GK, V, VW}. The first approach, which is carried out for the central limit Theorem~\ref{clt-main-cauchy-a2},
consists in some way of a copy of the corresponding proof of the CLT for $\beta$-Hermite ensembles in~\cite{V} and will be based on the explicit densities
(\ref{density-general-cauchy-intro}). The second
approach, which is carried out for the central limit Theorem~\ref{clt-main-cauchy-a2}, and which also works for $\beta$-Hermite ensembles, uses the
construction of the Cauchy--Bessel processes via subordination
and the known CLTs for Bessel processes from~\cite{V}. From a structural point of view, this second proof seems to be more natural;
however, the complexity of both proofs is about the same.

The Bessel processes of types A and B describe Calogero--Moser--Sutherland particle systems where the
parameters $k$ or $(k_1,k_2)$ correspond to inverse temperatures; see, e.g., \cite{DV}. Therefore our limits correspond to
freezing limits. Clearly, this interpretation is also available for the Cauchy--Bessel processes and distributions in this paper.

This paper is organized as follows. Section~\ref{section2} contains some background information on
Bessel and Cauchy--Bessel processes associated with root systems from \cite{An, CGY, R1,R2, RV1, RV2,DV}.
Sections~\ref{section3} and~\ref{section4} then are devoted to the
limit results for the root systems of types~A and~B respectively. We point out that besides the
central limit Theorems~\ref{clt-main-cauchy-a2} and~\ref{clt-main-cauchy-b} we also present a further asymptotic result in Theorem~\ref{clt-bessel-a}
where another norming of the given Cauchy--Bessel distributions of type~A is used and no weak convergence is available.
Furthermore, we briefly study the root systems of type~D in Section~\ref{section5};
this will be applied to some singular case for the root systems of type B there.

We finally point out that the Cauchy--Bessel ensembles in this paper are different from the Hua--Pickrell ensembles, which are studied, e.g., in
\cite{AD, As, ABGS, BO,H, N,P}, and which are also called Cauchy ensembles in some papers. However,
we expect that these Hua--Pickrell ensembles can be partially handled in a similar way as the Cauchy--Bessel ensembles in Section~\ref{section3} of this paper.

\section{Cauchy--Bessel processes}\label{section2}

Bessel processes associated with root systems can be used to describe
several integrable interacting particle systems of
Calogero--Moser--Suther\-land type on the real line $\mathbb R$ or $[0,\infty[$ with~$N$ particles; see for instance
\cite{An, CGY, R1, R2, RV1, RV2,DV} and references there for the background in analysis, probability, and mathematical physics.
We here mainly restrict our attention to the two most relevant classes, namely the root systems $A_{N-1}$ and $B_N$.
The root systems $D_N$ will be discussed briefly in Section~\ref{section5}.

In the cases $A_{N-1}$ and $B_N$, these processes are time-homogeneous diffusion processes
$(X_{t,k})_{t\ge0}$ living on the closed Weyl chambers
\[
C_N^A:=\big\{x\in \mathbb R^N\colon x_1\ge x_2\ge\dots\ge x_N\big\},\qquad C_N^B:=\big\{x\in \mathbb R^N\colon x_1\ge x_2\ge\dots\ge x_N\ge0 \big\}
\]
of types A and B. Here, $k$ is a parameter with
$k\subset[0,\infty[$ and $k=(k_1,k_2)\subset[0,\infty[^2$ for the root systems of types A and B respectively.
The generators of the transition semigroups are given by
\begin{gather}\label{def-L-A} L_Af:= \frac{1}{2} \Delta f +
 k \sum_{i=1}^N\Bigg( \sum_{j\colon j\ne i} \frac{1}{x_i-x_j}\Bigg) \frac{\partial}{\partial x_i}f \qquad\text{and}
 \\
\nonumber
 L_Bf:= \frac{1}{2} \Delta f +
 k_2 \sum_{i=1}^N \sum_{j\colon j\ne i} \left( \frac{1}{x_i-x_j}+\frac{1}{x_i+x_j} \right)
 \frac{\partial}{\partial x_i}f
 + k_1\sum_{i=1}^N\frac{1}{x_i}\frac{\partial}{\partial x_i}f, \end{gather}
where in both cases reflecting boundaries are assumed, i.e., the generators are applied to $C^2$-functions
which are invariant under the corresponding Weyl groups.

In both cases, the transition probabilities of the Bessel processes are given as follows; see \cite{R1,R2,RV1,RV2}.
For $t>0$, $x\in C_N$, $A\subset C_N$ a Borel set,
\begin{equation}\label{density-general}
K_t(x,A)=c_k \int_A \frac{1}{t^{\gamma_k+N/2}} {\rm e}^{-(\|x\|^2+\|y\|^2)/(2t)} J_k\left(\frac{x}{\sqrt{t}}, \frac{y}{\sqrt{t}}\right)
 w_k(y)\,{\rm d}y
\end{equation}
with the weights
\begin{equation*}
w_k^A(x):= \prod_{i,j\colon i<j}(x_i-x_j)^{2k}, \qquad
w_k^B(x):= \prod_{i,j\colon i<j}\big(x_i^2-x_j^2\big)^{2k_2} \prod_{i=1}^N x_i^{2k_1},\end{equation*}
the exponents
\begin{equation}\label{def-gamma}
\gamma_k^A(k)=kN(N-1)/2, \qquad \gamma_k^B(k_1,k_2)=k_2N(N-1)+k_1N,
\end{equation}
and the Selberg norming constants
\begin{equation}\label{const-A}
 c_k^A:= \bigg(\int_{C_N^A} {\rm e}^{-\|y\|^2/2} w_k^A(y) \,{\rm d}y\bigg)^{-1}
=\frac{N!}{(2\pi)^{N/2}} \prod_{j=1}^{N}\frac{\Gamma(1+k)}{\Gamma(1+jk)}
\end{equation}
and
\begin{align}
 c_k^B:={}& \bigg(\int_{C_N^B} {\rm e}^{-\|y\|^2/2} w_k^B(y) \,{\rm d}y\bigg)^{-1} \notag\\
={} &\frac{N!}{2^{N(k_1+(N-1)k_2-1/2)}} \prod_{j=1}^{N}\frac{\Gamma(1+k_2)}{\Gamma(1+jk_2)\Gamma\big(\frac{1}{2}+k_1+(j-1)k_2\big)},\label{const-B}
\end{align}
respectively. Notice that
$w_k$ is homogeneous of degree $2\gamma_k$.
Furthermore,
$J_k$ is a multivariate Bessel function of type $A_{N-1}$ or $B_N$ with multiplicities $k$ or $(k_1,k_2)$ respectively;
 see, e.g., \cite{R1,R2, RV1, RV2}.
We do not need much information about $J_k$. We only notice that
$J_k$ is analytic on $\mathbb C^N \times \mathbb C^N $ with
$J_k(x,y)>0$ for $x,y\in \mathbb R^N $.
Moreover, $J_k(x,y)=J_k(y,x)$ and $J_k(0,y)=1$
for $x,y\in \mathbb C^N $.
In particular, for the starting point $x=0\in C_N$, (\ref{density-general}) leads to the distributions
\begin{equation}\label{density-general-0}
c_k \frac{1}{t^{\gamma_k+N/2}} {\rm e}^{-\|y\|^2/(2t)}
 w_k(y)\,{\rm d}y,
\end{equation}
which are just the distributions of the $\beta$-Hermite and Laguerre ensembles from random matrix theory;
see, e.g.,~\cite{AGZ, D, DE1}.
In the last decade, several freezing limit theorems were derived for the distributions in~(\ref{density-general-0}) and also in
(\ref{density-general}) for general starting points for fixed $N$ and $k\to\infty$; see Dumitriu and Edelman~\cite{DE1}
for an approach via their tridiagonal random matrix models for $x=0$ and \cite{AHV, AKM1, AKM2,AV2, AV1, GK, V, VW}
for further limit results in this context.

In this paper we transfer some of these limit results
to Cauchy-type distributions.
To motivate these distributions we recapitulate the subordination procedure which leads from the Bessel processes $(X_{t,k})_{t\ge0}$
above
and the classical convolution semigroup of inverse Gaussian distributions
to Cauchy-type processes on the Weyl chambers from \cite{RV1}.
For this we consider the classical convolution semigroup $(\mu_t)_{t\ge0}$ of inverse Gaussian measures on $(\mathbb R,+)$
with $\mu_0=\delta_0$ and, for $t>0$,
\begin{equation}\label{inverse-gaussian}
{\rm d}\mu_t(s)=\frac{{\bf 1}_{]0,\infty[}(s)}{\sqrt{4\pi}}  t  s^{-3/2} \exp\big({-}t^2/(4s)\big)\,{\rm d}s ,
 \end{equation}
see, e.g., \cite[Section~9]{BF}. Moreover, let $(T_t)_{t\ge0}$ be an associated L\'evy process starting in~$0$
with c\`adl\`ag paths. Then the process
$(Y_t)_{t\ge0}$ with
$Y_t:= X_{T_t}$ for $t\ge0$ is a Feller process on $C_N$ whose transition probabilities are given by
\begin{equation}\label{general-subordination}
 Q_t(x,A) = \int_0^\infty K_s(x,A)\, {\rm d}\mu_t(s), \qquad t\ge0,\quad x\in C_N,\quad A\subset C_N.
 \end{equation}
As this construction is analogous to the classical construction of Cauchy processes from Brownian motions, we call
the processes $(Y_t)_{t\ge0}$ Cauchy--Bessel processes of type~A or~B, respectively.
It seems to be difficult to compute the densities of these distributions explicitly for general starting points $x\in C_N$ like
in~(\ref{density-general}) in terms of Bessel functions. On the other hand, for $x=0$ and $t>0$ one obtains that
the probability measures $Q_t(0,\cdot)$ have
the explicit Lebesgue densities
\begin{equation}\label{density-general-cauchy}
\frac{c_k t \Gamma(\gamma_k+(N+1)/2)}{\sqrt{4\pi}}
 \left(\frac{4}{ t^2+2\|y\|^2}\right)^{\gamma_k+(N+1)/2} w_k(y)
\end{equation}
with $c_k$, $\gamma_k$, $w_k$ as above depending on $k$ and the root system by some elementary calculus;
see also \cite[Section~5]{RV1} with a slightly different $t$-scaling.
In the next sections we study limits of these distributions for $k\to\infty$.
Due to the homogeneity property of these Cauchy--Bessel distributions w.r.t.~the scaling parameter~$t$,
we there restrict our attention to the case $t=\sqrt 2$ w.l.o.g.

We point out that for the root system $A_{N-1}$ and $k=1/2,1,2$, the densities (\ref{density-general-0})
admit the well-known interpretation as the distributions of the ordered
eigenvalues of Gaussian orthogonal, unitary, and symplectic
ensembles (GOE, GUE, GSE) respectively.
Therefore, the subordination above leading to the densities (\ref{density-general-cauchy})
corresponds to an analogous subordination of normal distributions on
the vector spaces associated with GOE, GUE, GSE, and the corresponding time normalizations.
Therefore, the densities (\ref{density-general-cauchy})
belong in these cases to Cauchy--Bessel distributions on these vector spaces where the entries of these matrices are no
longer independent.
A~similar interpretation exists for the root systems $B_N$ via subordinations of Laguerre ensembles.

\section[Limit theorems for the root system A\_\{N-1\}]{Limit theorems for the root system $\boldsymbol{A_{N-1}}$}\label{section3}

In this section we study the Cauchy--Bessel distributions with the densities (\ref{density-general-cauchy}) of type $A$ with
parameters $t=\sqrt 2$ and $k\ge0$. Taking the constants in~(\ref{def-gamma}), (\ref{const-A}) into account, we thus study
the distributions with the density
\begin{equation}\label{density-general-cauchy-a}
f_{k}(y):=C(k,N)
 \frac{1}{( 1+\|y\|^2)^{kN(N-1)/2+(N+1)/2}}  \prod_{i,j\colon i<j}(y_i-y_j)^{2k}
\end{equation}
on $C_N^A$ with the norming constant
\begin{equation}\label{density-general-cauchy-a-norming}
C(k,N) =
\frac{2^{kN(N-1)/2 }N! \Gamma(kN(N-1)/2+(N+1)/2)}{\pi^{(N+1)/2}} \prod_{j=1}^{N}\frac{\Gamma(1+k)}{\Gamma(1+jk)}.
\end{equation}

We first determine the maxima of $f_{k}$. In fact, as $f_{k}$ is equal to 0 on the boundary $\partial C_N^A$
with $\lim_{\|y\|\to\infty}f_{k}(y)=0$, $f_k$ has at least one maximum,
and all maxima are in the interior of~$C_N^A$. To determine these maxima,
we need the classical Hermite polynomials~$(H_N)_{N\ge 0}$ which are orthogonal w.r.t.\
the density ${\rm e}^{-x^2}$ on~$\mathbb R$. We normalize the $H_N$ as usual as, e.g., in \cite{S}
with
the three-term-recurrence
\begin{equation*}
H_0=1, \qquad H_1(x)=x, \qquad H_{n+1}(x)=2x H_n(x)-2nH_{n-1}(x), \qquad n\ge1.
\end{equation*}
Consider the vector
\[ {\bf z}=(z_{1},\dots, z_{N})\in C_N^A\]
whose entries are
 the ordered zeros of $H_N$.
We need the following known facts:

\begin{Lemma}\label{char-zero-A}
 For $N\ge2$ and $y\in C_N^A$, the following statements are equivalent:
\begin{enumerate}\itemsep=0pt
\item[\rm{(1)}] The function $\sum_{i,j\colon i<j} \ln(x_i-x_j) -\|x\|^2/2$ is maximal at $y\in C_N^A$;
\item[\rm{(2)}] For $i=1,\dots,N$: $y_i= \sum_{j\colon j\ne i} \frac{1}{y_i-y_j}$;
\item[\rm{(3)}] $y={\bf z}$.
\end{enumerate}
Furthermore,
\begin{equation}\label{square-sum-hermite-zeros} \sum_{i=1}^N z_{i}^2=\frac{N(N-1)}{2}\end{equation}
 and
\begin{equation}\label{potential-z-A}
 2\sum_{i<j} \ln(z_i-z_j)= -\frac{N(N-1)}{2} \ln 2+ \sum_{j=1}^N j\ln j.
\end{equation}
\end{Lemma}

\begin{proof}
 For the equivalence of (1)--(3) see \cite[Section~6.7]{S}; see also \cite{AKM1, V}. For (\ref{square-sum-hermite-zeros}) and
 (\ref{potential-z-A}) we refer to \cite[Appendix~D]{AKM1}; see in particular~(D.22) and~(D.30) there.
\end{proof}

Now let $x\in C_N^A$ be a maximum of $f_{k}$. This implies that $\nabla l(x)=0$
for
\[
l(y):= \ln\left( \frac{1}{( 1+\|y\|^2)^{kN(N-1)/2+(N+1)/2}} \prod_{i<j}(y_i-y_j)^{2k}\right).
\]
Therefore, for $i=1,\dots,N$,
\[
\sum_{j\colon j\ne i} \frac{1}{x_i-x_j}=\frac{kN(N-1)+N+1}{2k} \frac{1}{1+\|x\|^2} x_i,\]
i.e., $c_1 x_i= \sum_{j\colon j\ne i} \frac{1}{x_i-x_j}$ for $i=1,\dots,N$ with some constant $ c_1>0$.
A short computation now shows that for some constant $c_2>0$, the vector $y:=c_2 x$ satisfies the condition in Lemma~\ref{char-zero-A}(2)
and hence, by Lemma~\ref{char-zero-A}, $ {\bf z}=c_2 x$. In summary, with~(\ref{square-sum-hermite-zeros}) and a short computation, we obtain:

\begin{Lemma}\label{maximum-A}
The density $f_k$ has a unique maximum on $C_N^A$. This maximum is located in $\sqrt{\frac{2k}{N+1}} {\bf z}$.
\end{Lemma}

This elementary observation together with
the following known CLT for the densities (\ref{density-general-0}) of Bessel processes will be the motivation to study
the densities $f_k$ around these maxima for $k\to\infty$.

\begin{Theorem}\label{clt-bessel-a}
Let $X_{{\rm Bessel},k,N}$ be random variables with the densities \eqref{density-general-0} for the root system $A_{N-1}$ for
 $N\ge 2$. Then the random variables
$X_{{\rm Bessel},k,N} - \sqrt{2k} {\bf z}$
converge in distribution for \mbox{$k\to\infty$} to the $N$-dimensional centered normal distribution ${\mathcal N}(0,\Sigma)$ where the
 covariance matrix~$\Sigma$ is regular and has the following properties:
\begin{enumerate}\itemsep=0pt
\item[$(1)$] $\Sigma^{-1}=(s_{i,j})_{i,j=1,\dots,N}$ satisfies
\begin{equation}\label{covariance-matrix-A}
s_{i,j}:= \begin{cases} \displaystyle 1+\sum_{l\ne i} (z_{i,N}-z_{l,N})^{-2} & \text{for}\ i=j, \\
 -(z_{i,N}-z_{j,N})^{-2} & \text{for}\ i\ne j. \end{cases}
\end{equation}
\item[$(2)$] $\Sigma^{-1}$ has the eigenvalues $1,2,3,4,\dots,N$, and consequently $\det\Sigma^{-1}= N!$.
\item[$(3)$] $\Sigma=(\sigma_{i,j})_{i,j=1,\dots,N}$ satisfies
	\begin{equation}\label{covariance-matrix-A1}
	 \sigma_{i,j}=
 (-1)^{i+j}\frac{\sum_{k=0}^{N-1}\frac{H_k(z_{i,N})H_k(z_{j,N})}{2^kk!(N-k)}}{
\sqrt{\sum_{k=0}^{N-1}\frac{(H_k (z_{i,N}))^2}{2^kk!}\sum_{l=0}^{N-1}\frac{(H_l(z_{j,N}))^2}{2^ll!}}}.
	\end{equation}
\end{enumerate}
\end{Theorem}

Theorem \ref {clt-bessel-a} without parts~(1) and~(2) was obtained first by Dumitriu and Edelman
\cite{DE2} via their tridiagonal random matrices in~\cite{DE1} with different formulas for
the entries of $\Sigma$. It was then reproved in a
 direct way in \cite{V} with the entries of $\Sigma^{-1}$ in~(1).
Furthermore, in \cite{AV2} the eigenvalues and eigenvectors of $\Sigma^{-1}$ were determined,
and in \cite{AHV} the theory of dual orthogonal polynomials in the sense of De Boor and Saff
(see \cite{BS, I, VZ}) was used to obtain part~(3).
This CLT with part~(3) was also obtained in a different way by Gorin and Kleptsyn \cite{GK}.
We notice that it seems to be difficult to verify that the formulas in \cite{DE2}
 and those in part~(3) are identical,
as both formulas contain complicated expressions regarding the zeros of~$H_N$.

We now return to our Cauchy--Bessel distributions and try to copy the proof of the central limit Theorem~\ref{clt-bessel-a} from~\cite{V}.
It turns out that here a centering with the maxima from Lemma~\ref{maximum-A} does not lead to a full central limit theorem,
but to the following weaker asymptotic limit result only:

\begin{Theorem}\label{clt-main-cauchy-a}
 For $k>0$ and $N\ge2$ let $X_k$ be a $C_N^A$-valued random variable with density~$f_k$. Moreover, let~$\tilde f_k$
be the density of $X_k-\sqrt{\frac{2k}{N+1}} {\bf z}$. Then there is a unique centered normal distribution ${\mathcal N}(0,\Sigma_{\rm Cauchy})$ on $\mathbb R^N$
with some regular covariance matrix $\Sigma_{\rm Cauchy}$ and density $f$ such that
\begin{equation}\label{limit1-a}
\lim_{k\to\infty} \frac{\tilde f_k(x)}{ f(x)} k^{1/2}(N+1)^{N/2}\sqrt{N(N-1)} {\rm e}^{(N+1)/2} =1
\end{equation}
holds locally uniformly for $x\in\mathbb R^N$.
The matrix $\Sigma_{\rm Cauchy}$ has the following properties:
\begin{enumerate}\itemsep=0pt
\item[$(1)$] $\Sigma_{\rm Cauchy}^{-1}=(s_{i,j})_{i,j=1,\dots,N}$ satisfies
\begin{equation}\label{covariance-matrix-cauchy-A}
 s_{i,j}:=(N+1)\cdot
 \begin{cases}\displaystyle 1+\sum_{l\colon l\ne i} (z_i-z_l)^{-2} + \frac{4z_{i}^2}{N(N-1)}&
 \text{for}\ i=j, \\
 \displaystyle -(z_i-z_j)^{-2} + \frac{4z_{i}z_{j}}{N(N-1)}& \text{for} \ i\ne j. \end{cases}
\end{equation}
\item[$(2)$] $(N+1)^{-1}\Sigma_{\rm Cauchy}^{-1}$ has the eigenvalues $1,4, 3,4,5, \dots,N$, i.e., $\det\Sigma_{\rm Cauchy}^{-1}=2(N+1)^N N!$.
\item[$(3)$] $\Sigma_{\rm Cauchy}=(\sigma_{i,j})_{i,j=1,\dots,N}$ satisfies
\begin{equation*}
	 \sigma_{i,j}=
 \frac{(-1)^{i+j}}{N+1}\left( \frac{\sum_{k=0}^{N-1}\frac{H_k(z_{i,N})H_k(z_{j,N})}{2^kk!(N-k)}}{
\sqrt{\sum_{k=0}^{N-1}\frac{(H_k (z_{i,N}))^2}{2^kk!}\sum_{l=0}^{N-1}\frac{(H_l(z_{j,N}))^2}{2^ll!}}} - \frac{z_iz_j}{2N(N-1)}\right).
	\end{equation*}
 \end{enumerate}
\end{Theorem}

The proof of Theorem~\ref{clt-main-cauchy-a} is divided into two steps.
In a first step we show that (\ref{limit1-a}) holds
with $ \Sigma_{\rm Cauchy}^{-1}$ as part (1) up to a positive multiplicative constant in the limit.
 In the second step of the proof
 we then use Theorem~\ref{clt-bessel-a} and show that parts (2) and (3) hold, and that the constant in~(\ref{limit1-a})
is the correct one.

\begin{proof} Equation~(\ref{density-general-cauchy-a}) shows
 that the random variable $X_k-\sqrt{\frac{2k}{N+1}} {\bf z}$ has a density which can be written as
 \begin{equation*}
\tilde f_k(y)=f_{k}\left(y+\sqrt{\frac{2k}{N+1}} {\bf z}\right) =\tilde c_k {\rm e}^{h_k(y)}
 \end{equation*}
with the exponent
\begin{align}
 h_k(y):= {}&2k\sum_{i,j\colon i<j}\ln\left(1+ \frac{(y_i-y_j)\sqrt{N+1}}{\sqrt{2k}(z_i-z_j)}\right)\notag\\
 &{}-\frac{kN(N-1)+N+1}{2} \ln\left(1+ \frac{(1+\|y\|^2)(N+1)}{2k \|{\bf z}\|^2}+ 2\sqrt{\frac{N+1}{2k}}
 \frac{ \langle y,{\bf z}\rangle}{\|{\bf z}\|^2}\right)\label{a-density-detail-cauchy1}
\end{align}
and the constant
\begin{align}
\tilde c_k:={}&C(k,N) \exp\bigg( 2k\sum_{i,j\colon i<j}\ln\bigg(\sqrt{\frac{2k}{N+1}}(z_i-z_j)\bigg)\bigg)\notag\\
 & {}\times \exp\left({-} \frac{kN(N-1)+N+1}{2} \ln\left(\frac{2k}{N+1} \|{\bf z}\|^2\right)\right)
\notag\\
={} &C(k,N) \left(\frac{2k}{N+1}\right)^{kN(N-1)/2}\exp\bigg(2k \sum_{i,j\colon i<j}\ln(z_i-z_j)\bigg)
\notag\\ & {}\times \left(\frac{2k}{N+1}\|{\bf z}\|^2\right)^{-(kN(N-1)+N+1)/2}\label{a-density-detail-cauchy2}
\end{align}
on the shifted cone $C_N^A-\sqrt{2k} {\bf z}$ with $\tilde f_k(y)=0$ otherwise on $\mathbb R^N$.

We now study the exponent $h_k(y)$. The power series of $\ln(1+x)$ shows that for $k\to\infty$,
\begin{equation}\label{ln-expansion1}
\ln\bigg(1+ \frac{\sqrt{N+1}(y_i-y_j)}{\sqrt{2k}(z_i-z_j)}\bigg) = \frac{\sqrt{N+1}(y_i-y_j)}{\sqrt{2k}(z_i-z_j)}
 -\frac{(N+1)(y_i-y_j)^2}{4k(z_i-z_j)^2} + O\big(k^{-3/2}\big)
\end{equation}
and
\begin{align}
& \ln\bigg(1+\frac{(1+\|y\|^2)(N+1)}{2k \|{\bf z}\|^2}+ 2\sqrt{\frac{N+1}{2k}}
 \frac{ \langle y,{\bf z}\rangle}{\|{\bf z}\|^2}\bigg)\notag\\
 &\qquad{}= 2\sqrt{\frac{N+1}{2k}} \frac{ \langle y,{\bf z}\rangle}{\|{\bf z}\|^2}+
\frac{(1+\|y\|^2)(N+1)}{2k \|{\bf z}\|^2}- \frac{N+1}{k} \frac{ \langle y,{\bf z}\rangle^2}{\|{\bf z}\|^4}
+ O\big(k^{-3/2}\big).\label{ln-expansion2}
\end{align}
Moreover, by Lemma \ref{char-zero-A}(2),
\begin{equation}\label{zero-equation-a}
 \sum_{i,j\colon i<j} \frac{y_i-y_j}{z_i-z_j}- \langle y,{\bf z}\rangle=
\sum_{i=1}^N y_i\bigg( \sum_{j\colon j\ne i} \frac{1}{z_i-z_j}-z_i\bigg)=0.
\end{equation}
Therefore, by (\ref{a-density-detail-cauchy1}), (\ref{ln-expansion1})--(\ref{zero-equation-a}),
and (\ref{square-sum-hermite-zeros}),
\begin{equation*}
 h_k(y)=-\frac{N+1}{2}\biggl( \sum_{i,j\colon i<j}\frac{(y_i-y_j)^2}{(z_i-z_j)^2} + \big(1+\|y\|^2\big)
+\frac{4}{N(N-1)} \langle y,{\bf z}\rangle^2\biggr) + O\big(k^{-1/2}\big).
\end{equation*}
Therefore,
\begin{equation}\label{density-a-limit-eh}
 {\rm e}^{h_k(y)}\sim {\rm e}^{-(N+1)/2} \exp\bigl( {-}y^{\rm T} \Sigma_{\rm Cauchy}^{-1} y/2\bigr)
\end{equation}
with the matrix $\Sigma_{\rm Cauchy}^{-1}$ defined in~(\ref{covariance-matrix-cauchy-A}).

We next turn to the constants $\tilde c_k$ in (\ref{a-density-detail-cauchy2}). Here
(\ref{potential-z-A}) and (\ref{square-sum-hermite-zeros}) imply that
\begin{align*}
 \tilde c_k={}& C(k,N) \left(\frac{2k}{N+1}\right)^{kN(N-1)/2}\exp\bigg(k\bigg( {-}\frac{N(N-1)}{2}\ln 2 +\sum_{j=1}^N j\ln j
 \bigg)\bigg)
\notag\\& {}\times \left(\frac{N+1}{kN(N-1)}\right)^{kN(N-1)/2+(N+1)/2}
\notag\\
={} & C(k,N) \left(\frac{N+1}{k}\right)^{(N+1)/2} \left(\frac{1}{N(N-1)}\right)^{kN(N-1)/2+(N+1)/2}
 \prod_{j=1}^N j^{kj}.
\end{align*}
If we use (\ref{density-general-cauchy-a-norming}) and
Stirling's formula $\Gamma(k+1)\sim \sqrt{2\pi k}(k/{\rm e})^k$ for $k\to\infty$, an elementary, but tedious
calculation leads to
\begin{equation}\label{limit-constant-a}
\tilde c_k \sim k^{-1/2} \frac{\sqrt 2 \sqrt{N!}}{(2\pi)^{N/2} \sqrt{N(N-1)}}, \qquad k\to\infty,
\end{equation}
and thus with (\ref{density-a-limit-eh}) to
\begin{equation}\label{density-complete-a}
 \tilde f_k(y)\sim k^{-1/2} \frac{\sqrt 2 \sqrt{N!}}{(2\pi)^{N/2} \sqrt{N(N-1)} {\rm e}^{(N+1)/2}}
 \exp\bigl( {-}y^{\rm T} \Sigma_{\rm Cauchy}^{-1} y/2\bigr).
\end{equation}
An inspection of the preceding computations shows that (\ref{density-complete-a})
holds locally uniformly for \mbox{$y\in\mathbb R^N$}.
On the other hand,
${\mathcal N}(0,\Sigma_{\rm Cauchy})$ has the density
\begin{equation}\label{normal-density-a}
 f_(y)= \frac{1}{(2\pi)^{N/2} \sqrt{\det \Sigma_{\rm Cauchy} }}
 \exp\bigl( {-}y^{\rm T} \Sigma_{\rm Cauchy}^{-1} y/2\bigr).
\end{equation}
In order to determine $\det \Sigma_{\rm Cauchy} $, we
compare (\ref{covariance-matrix-A}) and (\ref{covariance-matrix-cauchy-A}) and use (\ref{square-sum-hermite-zeros}). We obtain that
 \begin{equation}\label{connection-covariance-a}
\Sigma_{\rm Cauchy}^{-1}= (N+1)\left( \Sigma^{-1}+ \frac{2}{\|{\bf z}\|^2}{\bf z}{\bf z}^{\rm T}\right),
\end{equation}
where, by Theorem \ref{clt-bessel-a}, $\Sigma^{-1}$ has the eigenvalues $1,2,3,\dots,N$. Moreover, by \cite{AV2},
${\bf z}$ is an eigenvector of $\Sigma^{-1}$ for the eigenvalue 2.
As the eigenvectors of a symmetric matrix are orthogonal, we conclude that
$(N+1)^{-1}\Sigma_{\rm Cauchy}^{-1}$ has the eigenvalues $1, 2+2=4,3,4,5,\dots,N$ where the eigenvectors are the same as for $\Sigma^{-1}$.
This proves part (2) of Theorem \ref{clt-main-cauchy-a} and yields that{\samepage
\[\det \Sigma_{\rm Cauchy} =\frac{1}{2(N+1)^N N!} .\]
This, (\ref{density-complete-a}), and (\ref{normal-density-a}) now lead to (\ref{limit1-a}).}

We finally turn to part (3) of Theorem \ref{clt-main-cauchy-a}. Using (\ref{connection-covariance-a}) and the fact that
$\Sigma_{\rm Cauchy}^{-1}$ and $\Sigma^{-1}$ have the same orthogonal transformation matrices $T$, we write these matrices as
\[
\Sigma_{\rm Cauchy}^{-1}= (N+1) T^{\rm T}\operatorname{diag}(1,4,3,4,\dots,N)T, \qquad \Sigma^{-1}= T^{\rm T}\operatorname{diag}(1,2,3,4,\dots,N)T.\]
Thus
\begin{align*}
 \Sigma_{\rm Cauchy}& =\frac{1}{N+1} T^{\rm T} \operatorname{diag}(1,1/4,1/3,1/4,\dots,1/N)T \\
 &=\frac{1}{N+1} \big( \Sigma- T^{\rm T} \operatorname{diag}(0,1/4,0,0,\dots,0)T\big)
 =\frac{1}{N+1} \left( \Sigma- \frac{1}{4\|{\bf z}\|^2}{\bf z}{\bf z}^{\rm T}\right).
\end{align*}
This and (\ref{covariance-matrix-A1}) now lead to part (3).
\end{proof}

Theorem \ref{clt-main-cauchy-a} shows that we need a stronger scaling of our Cauchy--Bessel distributions
(\ref{density-general-cauchy-a}) than in this theorem in order to obtain a weak limit result with a probability measure as limit.
We now study some suitable scaling where we use different
scales on two complementary subspaces of $\mathbb R^N$. To understand the idea, we first consider the case $N=2$.

\begin{Example}\label{ex-2}
 For $N=2$ we consider the densities $f_k(y)$ from (\ref{density-general-cauchy-a}) with the new orthogonal coordinates
\[x_1:=(y_1+y_2)/\sqrt 2, \qquad x_2:=(y_1-y_2)/\sqrt 2,\]
 i.e., $x_1\in\mathbb R$ describes the center of gravity and $x_2>0$ the distance between the two particles up to the precise scaling.
 By a short computation, in the new rotated coordinates, we then have the densities
 \begin{equation*}
 \frac{ 2^{2k+1}\Gamma(k+3/2)\Gamma(k+1)}{\pi^{3/2}\Gamma(2k+1)} \frac{x_2^{2k}}{\big(1+x_1^2+x_2^2\big)^{k+3/2}}
\end{equation*}
 for $x_2>0$ with the value $0$ otherwise. If we rescale the distance coordinate by $1/\sqrt k$,
 i.e., if we define a new coordinate $\tilde x_2:=x_2/\sqrt k$, we obtain a density which we write as
 \begin{equation*}
 \tilde f_k(x_1,\tilde x_2):= \frac{ 2^{2k+1}\Gamma(k+3/2)\Gamma(k+1)}{k \pi^{3/2}\Gamma(2k+1)}
 \left(1- \frac{1}{k} \frac{1+x_1^2}{\tilde x_2^2+\frac{1+x_1^2}{k}}\right)^k \left(\tilde x_2^2+\frac{1+x_1^2}{k}\right)^{-3/2}
 \end{equation*}
 for $x_1\in\mathbb R$ and $\tilde x_2>0$. By
 Stirling's formula $\Gamma(k+1)\sim \sqrt{2\pi k}(k/{\rm e})^k$, these densities tend to
 \begin{equation}\label{density-general-cauchy-a-limit2}
 f(x_1,\tilde x_2):=\frac{2}{\pi}{\rm e}^{-(1+x_1^2)/\tilde x_2^2} \frac{1}{\tilde x_2^3}
\end{equation}
 for $x_1\in\mathbb R$ and $\tilde x_2>0$ for $k\to\infty$. It can be easily checked that $f$ is in fact
 the density of a probability measure which has in the coordinate
 $\tilde x_2>0$ the image of the inverse Gaussian distribution $\mu_2$ from (\ref{inverse-gaussian}) with parameter $t=2$
under the mapping $\tilde x_2\mapsto \tilde x_2^2$ on $[0,\infty[$ as marginal distribution.
 If this is shown,
 it can be derived from (\ref{density-general-cauchy-a-limit2}) that in the coordinate $x_1$ a classical one-dimensional
 Cauchy distribution appears as marginal distribution.
\end{Example}

In particular, the classical Cauchy distribution as marginal distribution for the center-of-gravity-part is no accident,
and appears for all $N\ge2$. To explain this, we consider
 a diffusion process
 $(X_t:=(X_{t,1},\dots, X_{t,N}))_{t\ge0}$ on $C_N^A$
 associated with the generator (\ref{def-L-A}) with start in $0$ for $k>0$ and $N\ge2$.
 Moreover, for the vector ${ \bf 1}:=(1,\dots,1)\in\mathbb R^N$ we denote the orthogonal projections from
 $\mathbb R^N$ onto $\mathbb R\cdot {\bf 1 }$ and its orthogonal complement ${\bf 1 }^\perp$ by
 $p_{{\bf 1 }}$ and $p_{{\bf 1 }^\perp}$ respectively.
 We now consider the center-of-gravity-process
\[\big(X_t^{\rm cg}:=(X_{t,1}+\dots+ X_{t,N})/ N\big)_{t\ge0},\]
which may be regarded as $(p_{{\bf 1 }}(X_t))_{t\ge0}$ by identifying $x\in\mathbb R$ with $x\cdot{\bf 1 }$.
 It can be easily seen from (\ref{def-L-A}) (see for instance \cite[Lemma 3.2]{RV3} or \cite[Section~2]{AV2}) that
 this process is a usual one-dimensional Brownian motion
 (up to some scaling factor) which is stochastically independent from the orthogonal projection
\[\big( X_{t}^{\rm diff}:=p_{{\bf 1 }^\perp}(X_t)= X_{t}- X_t^{\rm cg}\cdot{\bf 1}\big)_{t\ge0}\]
 onto ${\bf 1 }^\perp$.
Notice
 that the $N-1$ coordinates of the diffusion $\big( X_{t}^{\rm diff}\big)_{t\ge0}$ describe the sucessive distances of the neighbored particles, and that
 the center-of-gravity-part $\big(X_t^{\rm cg}\big)_{t\ge0}$ is independent from~$k$.
Using our subordination procedure in~(\ref{inverse-gaussian}), (\ref{general-subordination}), and~(\ref{density-general-cauchy}) we thus obtain readily that the center-of-gravity marginal distributions
of the Cauchy--Bessel distributions with densities $f_k$ on $\mathbb R^N$ are a classical standard Cauchy distribution on $\mathbb R$
independent from~$k$.
In summary we obtain in this way:

\begin{Lemma}\label{lemma-one-dim-marginal}
Let $(X_{1},\dots, X_{N})$ be a $C_N^A$-valued random variable with the Lebesgue density $f_k$
from \eqref{density-general-cauchy-a} with $N\ge2$ and $k>0$.
Then $(X_{1}+\dots+ X_{N})/\sqrt N$ is standard Cauchy distributed
on $\mathbb R$ with the density $\frac{1}{\pi} \frac{1}{1+x^2}$.
\end{Lemma}

Therefore, also in the limit $k\to\infty$, a standard Cauchy distribution on $\mathbb R$ appears for
the center-of-gravity part.

Motivated by this result and Example~\ref{ex-2} for $N=2$, we now turn to some weak limit theorem for $N\ge2$.
We consider $C_N^A$-valued random variables $X_k$ with the densities $f_k$ as above.
Motivated by Example \ref{ex-2}, we now use different scalings on two complementary subspaces of $\mathbb R^N$,
namely the one-dimensional subspace $\mathbb R\cdot {\bf z }$ and its orthogonal complement $ {\bf z }^\perp$.
Let $p_{{\bf z }}\colon \mathbb R^N\to \mathbb R\cdot {\bf z }$ be the orthogonal projection onto $\mathbb R\cdot {\bf z }$
which satisfies
\begin{equation*}
p_{{\bf z }}(y)=\frac{\langle y,{\bf z }\rangle }{\|{\bf z }\|^2}{\bf z }=\frac{2\langle y,{\bf z }\rangle}{N(N-1)}{\bf z }.
\end{equation*}
Moreover, the orthogonal projection onto ${\bf z }^\perp$ is given by
\begin{equation*}
p_{{\bf z }^\perp}(y)=y-\frac{\langle y,{\bf z }\rangle }{\|{\bf z }\|^2}{\bf z }=y-\frac{2\langle y,{\bf z }\rangle}{N(N-1)}{\bf z }.
\end{equation*}
We now define the rescaled random variables
\begin{equation}\label{def-x-tilde}
\tilde X_k:= \phi_k(X_k)
\end{equation}
with the linear mappings
\[\phi_k\colon \ \mathbb R^N\to \mathbb R^N, \qquad \phi_k(y):=
\frac{1}{\sqrt k} p_{{\bf z }}(y) +p_{{\bf z }^\perp}(y)= y+\left(\frac{1}{\sqrt k}-1\right)p_{{\bf z }}(y)\]
for $k\ge1$. The random variables $\tilde X_k$ then have values in the sets $C_{N,k}^A:=\phi_k\big(C_N^A\big)$.
These sets have the following properties:

\begin{Lemma}\label{support-a}
 The closure of $\bigcup_{k\ge1}C_{N,k}^A$ is the closed half space
\[B_N:=\big\{y\in \mathbb R^N\colon \langle y,{\bf z }\rangle\ge0\big\}.\]
 Moreover, for $1\le k_1\le k_2$, $C_{N,k_1}^A\subset C_{N,k_2}^A$.
\end{Lemma}

\begin{proof}
 For $y\in C_{N,k}^A$ we have
\[\langle \phi_k(y),{\bf z }\rangle=\langle y,{\bf z }\rangle+ \left(\frac{1}{\sqrt k}-1\right)\langle y,{\bf z }\rangle=
 \frac{1}{\sqrt k}\langle y,{\bf z }\rangle\ge0\]
 and thus $\bigcup_{k\ge1}C_{N,k}^A\subset B_N$.
 For the converse statement we first consider some $w$ in the interior of $B_N$, i.e., with
 $\langle w,{\bf z }\rangle>0$. We now choose some $t>0$ sufficiently large with $y:= w+tz\in C_N^A$.
 Notice that this is possible for any vector $w$, as ${\bf z }$ is in the interior of $C_N^A$.
 Then we obtain for all $k\ge1$ that
\[\phi_k(y)= w + \left( \frac{t}{\sqrt k}+\left(\frac{1}{\sqrt k}-1\right)\langle w,{\bf z }\rangle\right){\bf z }.\]
 Therefore, if we take the unique $k=k(t)\ge1$ with
 $t=\big(\sqrt k-1\big)\frac{\langle w,{\bf z }\rangle}{\|{\bf z }\|^2}>0$, we obtain $\phi_k(y)=w$.
 We thus conclude that the interior of $B_N$ is contained in $\bigcup_{k\ge1}C_{N,k}^A$.
 This completes the proof of the first statement of the lemma.
 The second statement can be checked in a similar way.
\end{proof}

With the first statement of Lemma \ref{support-a} on the ranges of the random variables $X_k$ in mind, we now turn to the
following limit theorem.

\begin{Theorem}\label{clt-main-cauchy-a2}
 For $k>0$ and $N\ge2$ let $X_k$ be a $C_N^A$-valued random variable with density~$f_k$.
 Then the $\mathbb R^N$-valued rescaled random variables $\tilde X_k$ from~\eqref{def-x-tilde}
 converge in distribution for \mbox{$k\to\infty$} to some probability measure $\mu\in M^1\big(\mathbb R^N\big)$ with~$B_N$
 as support. $\mu$ has the Lebesgue density
\begin{align}
 f(y):={}& \frac{\sqrt{N!} (N(N-1))^{N/2}{\rm e}^{N(N-1)}}{\pi^{N/2} 2^{(N-1)/2}} \exp\Biggl({-}\frac{\|{\bf z }\|^2}{\| p_{{\bf z }}(y)\|^2}\Biggl(
\sum_{i,j\colon i<j}\frac{( y_i- y_j)^2}{(z_i-z_j)^2} +\|y\|^2\Biggr)\Biggr)
 \nonumber\\
 &{}\times\exp\left(\frac{- N(N-1)}{2 \|p_{{\bf z }}(y)\|^2}\right) \frac{1}{ \| p_{{\bf z }}(y)\|^{N+1}}\label{hka-final2-theorem}
\end{align}
for $y$ in the interior of the half space $B_N$.
\end{Theorem}

A short calculation shows that for $N=2$, the measure $\mu$ with density (\ref{hka-final2-theorem}) is in fact equal to the
limit in Example \ref{ex-2} where one has to take the rotation in the coordinates there into account.

The proof of Theorem \ref{clt-main-cauchy-a2} will be decomposed into two parts. In the first part we show that distributions
of the random variables $X_k$ converge vaguely to the measure $\mu$ on $\mathbb R^N$ with
the density $f$ from (\ref{hka-final2-theorem}) on the interior of $B_N$.
In a second step we then check that the density $f$ from (\ref{hka-final2-theorem})
is in fact the density of a probability measure, which then implies weak convergence.

\begin{proof}[First part of the proof of Theorem \ref{clt-main-cauchy-a2}]
 Let $k\ge1$. By using our rescaling on the one-dimensional subspace $\mathbb R\cdot {\bf z }$ together
 with the transformation formula for the densities of transformed random variables,
 the random variable $\tilde X_k$ has the density
\[ \tilde f_k(y):= \sqrt k f_k(\tilde y) \quad\quad\text{with} \quad
 \tilde y:={\sqrt k} p_{{\bf z }}(y) +p_{{\bf z }^\perp}(y)=y+\big(\sqrt k-1\big)p_{{\bf z }}(y)\]
 on the interior of $C_{N,k}$.
Using (\ref{density-general-cauchy-a}) we write this density as
\begin{gather}\label{mod-density-a1}
\tilde f_k(y)=\sqrt k C(k,N)
\Biggl( \prod_{i,j\colon i<j}\!\frac{(\tilde y_i-\tilde y_j)^{2}}{ \|\tilde y\|^2}\Biggr)^k\!
\left(\frac{\|\tilde y\|^2}{1+ \|\tilde y\|^2}\right)^{kN(N-1)/2} \!\left(\frac{1}{1+ \|\tilde y\|^2}\right)^{(N+1)/2}.\!\!\!
 \end{gather}
We next notice that
\[\|\tilde y\|^2=k \| p_{{\bf z }}(y)\|^2+\| p_{{\bf z }^\perp}(y)\|^2=\| y\|^2 +(k-1)\| p_{{\bf z }}(y)\|^2.\]
Hence, for $k\to\infty$,
\begin{align}
\left(\frac{\|\tilde y\|^2}{1+ \|\tilde y\|^2}\right)^{kN(N-1)/2} &=
\left(1- \frac{1}{k ( \| p_{{\bf z }}(y)\|^2+(1+\| p_{{\bf z }^\perp}(y)\|^2)/k)}\right)^{kN(N-1)/2}\notag \\
&\to \exp\left(\frac{- \|{\bf z }\|^2}{ \|p_{{\bf z }}(y)\|^2}\right)\label{mod-density-a2}
 \end{align}
and
\begin{equation}\label{mod-density-a3}
 \left(\frac{1}{1+ \|\tilde y\|^2}\right)^{(N+1)/2} \sim \frac{1}{ k^{(N+1)/2}\| p_{{\bf z }}(y)\|^{N+1}}.
 \end{equation}
Furthermore, we write the remaining term in (\ref{mod-density-a1}) as
\begin{equation}\label{mod-density-a4}
 \Biggl(\prod_{i,j\colon i<j}\frac{(\tilde y_i-\tilde y_j)^{2}}{ \|\tilde y\|^2}\Biggr)^k= {\rm e}^{h_k(y)}\end{equation}
 with
\begin{align*}
 h_k(y):={}& 2k \sum_{i,j\colon i<j}\ln(\tilde y_i-\tilde y_j) -\frac{kN(N-1)}{2}\ln\big(\|\tilde y\|^2\big)\notag\\
 ={}& 2k\sum_{i,j\colon i<j}\ln\left(\frac{(\sqrt{k}-1) \langle y,{\bf z }\rangle }{\|{\bf z }\|^2} (z_i-z_j) + y_i- y_j\right)
 \notag\\
 &{} -\frac{kN(N-1)}{2}\ln\big(
 \| y\|^2 +(k-1)\| p_{{\bf z }}(y)\|^2 \big). 
\end{align*}
 We now use $\langle y,{\bf z }\rangle>0$ and write $h_k(y)$ as
\begin{align}
 h_k(y)={}& 2k \sum_{i,j\colon i<j}\ln\left(1+ \frac{\|{\bf z }\|^2( y_i- y_j)}{\sqrt{k} \langle y,{\bf z }\rangle (z_i-z_j)} -
 \frac{1}{\sqrt k} \right)\notag \\
 &{}-\frac{kN(N-1)}{2}\ln\left(1+
\frac{1}{k}\left(\frac{ \| y\|^2}{\| p_{{\bf z }}(y)\|^2}-1\right) \right) +R_k\label{mod-density-a6}
\end{align}
with
\begin{align*}
 R_k:={}& 2k \sum_{i,j\colon i<j}\ln\left(\frac{\sqrt{k} \langle y,{\bf z }\rangle (z_i-z_j)}{\|{\bf z }\|^2} \right)
 -\frac{kN(N-1)}{2}\ln\left( \frac{k \langle y,{\bf z }\rangle^2}{\|{\bf z }\|^2} \right)\\
 ={}&2k \sum_{i,j\colon i<j}\ln\left(\frac{\sqrt{k} (z_i-z_j)}{\|{\bf z }\|^2} \right)
 -\frac{kN(N-1)}{2}\ln\left( \frac{k}{\|{\bf z }\|^2} \right).
\end{align*}
This, (\ref{potential-z-A}), (\ref{square-sum-hermite-zeros}), and elementary calculus now lead to
\begin{equation}\label{mod-density-a8}
 {\rm e}^{R_k}= (N(N-1))^{-kN(N-1)/2} \prod_{j=1}^N j^{kj}.
\end{equation}
Moreover, the power series
of $\ln(1+x)$ for the logarithms in (\ref{mod-density-a6}) shows for $k\to\infty$ that
\begin{align}
& \ln\bigg(1+ \frac{\|{\bf z }\|^2( y_i- y_j)}{\sqrt{k} \langle y,{\bf z }\rangle (z_i-z_j)} -
 \frac{1}{\sqrt k }\bigg)\notag
 \\ &\qquad{} =
 \frac{1}{\sqrt k}
 \left(\frac{\|{\bf z }\|^2( y_i- y_j)}{ \langle y,{\bf z }\rangle (z_i-z_j)} -1\right) -
 \frac{1}{2k}\left(\frac{\|{\bf z }\|^2( y_i- y_j)}{\langle y,{\bf z }\rangle (z_i-z_j)} -1\right)^2
+ O\big(k^{-3/2}\big)\label{ln-expansion3}
\end{align}
and
\begin{equation}\label{ln-expansion4}
\ln\left(1+ \frac{1}{k}\left(\frac{ \| y\|^2}{\| p_{{\bf z }}(y)\|^2}-1\right) \right)
 = \frac{1}{k}\left(\frac{ \| y\|^2}{\| p_{{\bf z }}(y)\|^2}-1\right) +O\big(k^{-2}\big).
\end{equation}
We next use (\ref{zero-equation-a}) and (\ref{square-sum-hermite-zeros}) and observe that
\begin{equation}\label{sum-zero}
 \sum_{i,j\colon i<j}\left(\frac{\|{\bf z }\|^2( y_i- y_j)}{\langle y,{\bf z }\rangle (z_i-z_j)} -1\right)=0.
\end{equation}
In summary, we conclude from (\ref{mod-density-a6}), (\ref{ln-expansion3}), (\ref{ln-expansion4}), and (\ref{sum-zero}) that
\begin{align}
 h_k(y)={}& - \sum_{i,j\colon i<j} \left(\frac{\|{\bf z }\|^2( y_i- y_j)}{\langle y,{\bf z }\rangle (z_i-z_j)} -1\right)^2
 -\frac{N(N-1)}{2}\left(\frac{ \| y\|^2}{\| p_{{\bf z }}(y)\|^2}-1\right)
+R_k+O\big(k^{-1/2}\big)\notag\\
={}&-\sum_{i,j\colon i<j}\frac{\|{\bf z }\|^4( y_i- y_j)^2}{\langle y,{\bf z }\rangle^2 (z_i-z_j)^2} -\frac{N(N-1)}{2}+2
\sum_{i,j\colon i<j}\frac{\|{\bf z }\|^2( y_i- y_j)}{\langle y,{\bf z }\rangle (z_i-z_j)} \notag\\
&{}-\|{\bf z }\|^2 \left(\frac{ \| y\|^2\| {\bf z}\|^2}{\|\langle y,{\bf z }\rangle^2}-1\right)
+R_k+O\big(k^{-1/2}\big)\notag\\
={}&-\sum_{i,j\colon i<j}\!\!\frac{\|{\bf z }\|^4( y_i- y_j)^2}{\langle y,{\bf z }\rangle^2 (z_i-z_j)^2} +\frac{N(N-1)}{2}
-\frac{\|{\bf z }\|^4\|y\|^2}{\langle y,{\bf z }\rangle^2}+\|{\bf z }\|^2
+R_k+O\big(k^{-1/2}\big).\!\!\!\!\!\!\!\label{hka-final}
\end{align}
We next consider the constant $C(k,N)$ from (\ref{density-general-cauchy-a-norming}).
Stirling's formula $\Gamma(k+1)\sim \sqrt{2\pi k}(k/{\rm e})^k$ for $k\to\infty$, and an elementary, but tedious
calculation as in (\ref{limit-constant-a}) leads to
\begin{equation*}
 C(k,N)\sim \frac{\sqrt{N!} k^{N/2} (N(N-1))^{kN(N-1)/2+N/2}}{\pi^{N/2} 2^{(N-1)/2} \prod_{j=1}^{N} j^{kj}},
 \qquad k\to\infty.
\end{equation*}
This, (\ref{mod-density-a1}), (\ref{mod-density-a2}), (\ref{mod-density-a3}), (\ref{mod-density-a4}), (\ref{mod-density-a8}),
and (\ref{hka-final}) now show that
\begin{equation*}
 \lim_{k\to\infty} \tilde f_k(y)= f(y)
\end{equation*}
for the density $f$ from (\ref{hka-final2-theorem}) and
for $y$ with $\langle y,{\bf z }\rangle>0$, i.e., for $y$ in the interior of the half space~$B_N$.
We also observe by inspection of the preceding arguments that the convergence above holds locally uniformly in $y$
in the interior of the half space $B_N$. We thus conclude that the distributions of the random variables $\tilde X_k$ tend vaguely
to the measure $\mu$ with density $f$ in the interior of $B_N$.
\end{proof}

\begin{proof}[Second part of the proof of Theorem \ref{clt-main-cauchy-a2}]
 In order to complete the proof of Theorem~\ref{clt-main-cauchy-a2} we now check that the density $f$ with $f(y):=0$
 for $y$ on the boundary of $B_N$, i.e., for $y$ with $\langle y,{\bf z }\rangle=0$ is in fact the density of
 a probability measure.
 If this is shown, a standard argument in probability applied to the interior of $B_N$ then implies weak convergence as claimed.

 In order to compute $\int_{B_N} f(y)\,{\rm d}y$, we first observe that
 \begin{equation*}
 \exp\bigg({-}\frac{\|{\bf z }\|^2}{\| p_{{\bf z }}(y)\|^2}\bigg(
 \sum_{i,j\colon i<j}\frac{( y_i- y_j)^2}{(z_i-z_j)^2} +\|y\|^2\bigg)\bigg)
 = \exp\bigg({-}\frac{\|{\bf z }\|^2}{\| p_{{\bf z }}(y)\|^2}y^{\rm T}\Sigma^{-1}y\bigg)
\end{equation*}
with the matrix $\Sigma^{-1}=(s_{i,j})_{i,j=1,\dots,N}$ defined in (\ref{covariance-matrix-A}) where
$\Sigma^{-1}$ has the eigenvalues $1,2,3,4,\allowbreak \dots,N$ by Theorem~\ref{clt-bessel-a}.
We also recapitulate that the vectors ${\bf 1}$, ${\bf z}$ are eigenvectors of $\Sigma^{-1}$ for the eigenvalues 1,~2 respectively. We now choose an
orthonormal basis of~$\mathbb R^N$ consisting of eigenvectors of $\Sigma^{-1}$ associated with the eigenvalues $1,3,4,5,\dots,N$
and $2$ respectively where we choose the vector $ {\bf z}/\| {\bf z}\|$ as ``the'' eigenvector associated with the eigenvalue 2.
Moreover, for $y=(y_1,\dots,y_{N-1})\in \mathbb R^{N-1}$ and $t\in \mathbb R$ let $(y,t):=(y_1,\dots,y_{N-1},t)\in \mathbb R^{N}$.
With this notation and the constant
\[D(N):=\frac{\sqrt{N!} (N(N-1))^{N/2}{\rm e}^{N(N-1)}}{\pi^{N/2} 2^{(N-1)/2}}\]
from the density in (\ref{hka-final2-theorem}), we obtain by elementary calculations with a orthogonal transformation and with
the norming of the inverse Gaussian density in
(\ref{inverse-gaussian}) that
\begin{align*}
\int_{B_N}f(y)\,{\rm d}y={}& D(N) \int_0^\infty\Biggl(\int_{\mathbb R^{N-1}}
\exp\left(-\frac{\|{\bf z }\|^2}{t^2}(y,t)^{\rm T} \operatorname{diag}(1,3,4,5,\dots,N,2)(y,t)\right){\rm d}y \Biggr) \\
&{}\times\exp\left(\frac{- N(N-1)}{2t^2}\right) \frac{1}{ t^{N+1}}\,{\rm d}t \\
={}&D(N) {\rm e}^{- N(N-1)} \int_0^\infty\Biggl(\int_{\mathbb R^{N-1}}
\exp\left(-\frac{\|{\bf z }\|^2}{t^2}y^{\rm T} \operatorname{diag}(1,3,4,5,\dots,N)y\right){\rm d}y \Biggr) \\
&{}\times
\exp\left(\frac{- N(N-1)}{2t^2}\right) \frac{1}{ t^{N+1}}\,{\rm d}t\\
={}&D(N) {\rm e}^{- N(N-1)} \frac{(2\pi)^{N-1)/2}\sqrt 2}{(N(N-1))^{(N-1)/2}\sqrt{N!}}
 \int_0^\infty \exp\left(\frac{- N(N-1)}{2t^2}\right) \frac{1}{ t^{2}}\,{\rm d}t =1
\end{align*}
as claimed. This completes the proof.
\end{proof}

\begin{Remark}\quad
 \begin{enumerate}\itemsep=0pt
 \item[\rm{(1)}] Let $X$ be a random variable with values in the half space $B_N$ and density (\ref{hka-final2-theorem}) as in Theorem
 \ref{clt-main-cauchy-a2}. Then $(X_{1}+\dots+ X_{N})/\sqrt N$ is standard Cauchy distributed
 on $\mathbb R$. This follows immediately from Lemma \ref{lemma-one-dim-marginal} and the fact that the maps $\phi_k$ leave $x_{1}+\dots+ x_{N}$
 invariant for each $x\in\mathbb R^N$ because of $\langle {\bf 1}, {\bf z}\rangle=0$.
 \item[\rm{(2)}] There is a second, more structural proof of Theorem~\ref{clt-main-cauchy-a2} which explains the limit density~(\ref{hka-final2-theorem})
 in terms of subordination; see Remark~\ref{b-final-remarks}(1) below.
 \end{enumerate}
\end{Remark}

\section[Limit theorems for the root system B\_N]{Limit theorems for the root system $\boldsymbol{B_N}$}\label{section4}

We now study the Cauchy--Bessel distributions with the densities (\ref{density-general-cauchy}) of type $B$ with the
parameters $t=\sqrt 2$ and $k=(k_1,k_2)$ with $k_1,k_2\ge0$. Following \cite{AKM2, AV2, AV1, V} we write $k$ as
$(k_1,k_2)=(\nu \beta,\beta)$ where we fix $\nu>0$ and investigate limits for $\beta\to\infty$.
Taking these new parameters and the constants in (\ref{def-gamma}) and (\ref{const-B}) into account, we thus study
the distributions~$\tau_{\nu,\beta}$ with the density
\begin{equation*}
f_{\nu,\beta}(y):=C_B(\nu,\beta,N)
 \frac{1}{\big( 1+\|y\|^2\big)^{\beta N(N+\nu -1)+(N+1)/2}} \prod_{i,j\colon i<j}\big(y_i^2-y_j^2\big)^{2\beta}\prod_{i=1}^N y_i^{2\nu\beta}
\end{equation*}
on the Weyl chambers $C_N^B$ with the norming constants
\begin{gather}
C_B(\nu,\beta,N) =
\frac{N! 2^{N} \Gamma(\beta N(N+\nu-1) + (N+1)/2)}{\sqrt\pi }\nonumber\\
\hphantom{C_B(\nu,\beta,N) =}{}\times
 \prod_{j=1}^{N}\frac{\Gamma(1+\beta)}{\Gamma(1+j\beta)\Gamma(\frac{1}{2}+\beta(j+\nu-1))}.\label{density-general-cauchy-b-norming}
\end{gather}
We now proceed as in Section~\ref{section3} and use the Laguerre polynomials
$L_N^{(\nu-1)}$ instead of the Hermite polynomials. Recapitulate that the $L_N^{(\nu-1)}$ are orthogonal w.r.t.\
the density ${\rm e}^{-x} x^{\nu-1}$ on $]0,\infty[$
for $\nu>0$ as defined in~\cite{S}.
We recapitulate the following facts about the zeros of $L_N^{(\nu-1)}$.

\begin{Lemma}\label{char-zero-B1}
Let $\nu>0$. For $r=(r_1,\dots,r_N)\in C_N^B$, the following statements are equivalent:
\begin{enumerate}\itemsep=0pt
\item[$(1)$] The function
\[ W_B(y):=2\sum_{ i<j} \ln\big(y_i^2-y_j^2\big) +2\nu \sum_{i}\ln y_i-\|y\|^2/2\]
 is maximal at $r\in C_N^B$;
\item[$(2)$] For $i=1,\dots,N$,
$\frac{1}{2}r_i= \sum_{j\colon j\ne i} \frac{2r_i}{r_i^2-r_j^2} +\frac{\nu}{r_i}=\sum_{j: j\ne i} \big(
\frac{1}{r_i-r_j} +\frac{1}{r_i+r_j}\big) +\frac{\nu}{r_i} $;
\item[$(3)$] If $z_1^{(\nu-1)}\ge\dots\ge z_N^{(\nu-1)}$ are the ordered zeros of $L_N^{(\nu-1)}$, then
\begin{equation*}
2\big(z_1^{(\nu-1)},\dots, z_N^{(\nu-1)}\big)= \big(r_1^2, \dots, r_N^2\big).
\end{equation*}
\end{enumerate}
The vector $r$ of $(1)$--$(3)$ satisfies
\begin{equation}\label{equality-F-B0}
 \|r\|^2=N(N+\nu-1)
\end{equation}
and
\begin{align}
&-\frac{1}{2}\|r\|^2 +\nu\sum_{j=1}^N \ln r_j^2 + 2\sum_{i<j}\ln \big( r_i^2 -r_j^2\big) \notag \\
&\qquad{} =N(N+\nu-1)(-1+\ln 2)+ \sum_{j=1}^N j\ln j + \sum_{j=1}^N(\nu +j-1) \ln(\nu +j-1).\label{equality-F-B}
\end{align}
\end{Lemma}

\begin{proof} See \cite{AKM2} or \cite{V}. Parts are also in~\cite[Section~6.3]{S}.\end{proof}

This result leads to the following CLT in the Bessel case; see \cite[Theorem 3.3]{V}.

\begin{Theorem}\label{clt-main-b1}
 Let $\nu>0$, $N\ge1 $ an integer, and $(X_{t,\beta})_{t\ge0}$ a Bessel process of type $B_N$ on $C_N^B$ starting in $0\in C_N^B$
with parameter $k=(\nu \beta,\beta)$.
Then, for the vector $r\in C_N^B$ from Lemma~{\rm \ref{char-zero-B1}},
\[\frac{X_{t,\beta}}{\sqrt t} - \sqrt{\beta } r\]
converges for $\beta\to\infty$ to the centered $N$-dimensional distribution ${\mathcal N}(0, \Sigma)$
with the regular covariance matrix $\Sigma$ with $\Sigma^{-1}=(s_{i,j})_{i,j=1,\dots,N}$ with
\begin{equation*}
s_{i,j}:= \begin{cases}\displaystyle 1+ \frac{2\nu}{r_i^2}+2\sum_{l\ne i} (r_i-r_l)^{-2}+2\sum_{l\ne i} (r_i+r_l)^{-2} &
 \text{for}\ i=j, \\
 2(r_i+r_j)^{-2} -2(r_i-r_j)^{-2} & \text{for}\ i\ne j. \end{cases}
\end{equation*}
The matrix $\Sigma^{-1}$ has the eigenvalues $2,4,\dots, 2N$.
\end{Theorem}

We now derive an associated weak limit law for the Cauchy--Bessel distributions $\tau_{\nu,\beta}$ analogous to
 Theorem \ref{clt-main-cauchy-a2}.
We here again use different scalings on two complementary subspaces of $\mathbb R^N$,
namely on $\mathbb R\cdot r$ and its orthogonal complement $r^\perp$.
Let $p_{r}\colon \mathbb R^N\to \mathbb R\cdot r$ be the orthogonal projection onto $\mathbb R\cdot r$,
and $p_{{r}^\perp}$ the orthogonal projection onto ${r}^\perp$.
Now let $X_{\beta}$ be a $\tau_{\nu,\beta}$-distributed random variable.
We again define the rescaled random variables
$\tilde X_\beta:= \phi_\beta( X_\beta)$
with the linear mappings
\begin{equation}\label{lin-mp-b}
 \phi_\beta\colon \ \mathbb R^N\to \mathbb R^N, \qquad \phi_\beta(y):=
 \frac{1}{\sqrt \beta} p_{r}(y) +p_{{r}^\perp}(y)= y+\left(\frac{1}{\sqrt\beta }-1\right)p_{{r}}(y).
 \end{equation}
 The random variables $\tilde X_\beta$ then have values in the sets $C_{N,\beta}^B:=\phi_\beta\big(C_N^B\big)$.
These sets have the following property analogous to Lemma~\ref{support-a}.

\begin{Lemma}\label{support-b}
 The closure of $\bigcup_{\beta\ge1}C_{N,\beta}^B$ is the closed half space
\[B_N:=\big\{y\in \mathbb R^N\colon \langle y,{r }\rangle\ge0\big\}.\]
 Moreover, for $1\le\beta_1\le \beta_2$, $C_{N,\beta_1}^B\subset C_{N,\beta_2}^B$.
\end{Lemma}

The following limit theorem is analogous to Theorem \ref{clt-main-cauchy-a2}.

\begin{Theorem}\label{clt-main-cauchy-b}
 For $\beta>0$ and $N\ge1$ let $X_\beta$ be a $C_N^A$-valued, $\tau_{\nu,\beta}$-distributed random variable.
 Then the rescaled random variables $\tilde X_\beta$
 converge in distribution for $\beta\to\infty$ to some probability measure $\mu$ on $\mathbb R^N$ with
 the half space $B_N$ as support. This measure $\mu$ is given by
 \begin{equation}\label{hkb-final2-theorem-intrep}
 \mu:=\frac{1}{\sqrt{2\pi}}\int_0^\infty {\mathcal N}\big(\sqrt s r,s A\Sigma A\big) s^{-3/2}\exp(-1/(2s))\,{\rm d}s
 \end{equation}
 with the matrix $\Sigma$ from Theorem~{\rm \ref{clt-main-b1}}
 where $ A$ is the matrix belonging to the orthogonal projection~$p_{{r}^\perp}$ in the standard coordinates on $\mathbb R^N$.
Moreover, $\mu$ has the Lebesgue density
\begin{gather}\label{hkb-final2-theorem}
 f(y):= D(N)\exp\left(-\frac{N(N+\nu-1)}{2\| p_{{r }}(y)\|^2} y^{\rm T} \Sigma^{-1} y \right)
 \exp\left(\frac{- N(N+\nu-1)}{ 2 \|p_{{r }}(y)\|^2}\right) \frac{1}{ \| p_{{r}}(y)\|^{N+1}}
 \end{gather}
with
\[D(N):=\frac{\sqrt{2} \sqrt{N!} (N(N+\nu-1))^{N/2}{\rm e}^{N(N+\nu-1)}}{\pi^{N/2} }\]
for $y$ in the interior of~$B_N$.
\end{Theorem}

Please notice that the normal distributions in the mixing formula (\ref{hkb-final2-theorem-intrep}) are singular,
and that the existence of the density (\ref{hkb-final2-theorem}) on the half space $B_N$ is a consequence of the integration
w.r.t.\ the mean vectors of the normal distributions which compensate the singular direction.

Theorem \ref{clt-main-cauchy-b} with the limit with density (\ref{hkb-final2-theorem})
can be derived in the same way as Theorem~\ref{clt-main-cauchy-a2} by using Lemmas~\ref{char-zero-B1} and~\ref{support-b}. We skip this direct approach and present a second, Fourier-analytic
proof which is based on the central limits Theorem~\ref{clt-main-b1} and the very construction
of the Cauchy--Bessel distributions $\tau_{\nu,\beta}$ via the subordination~(\ref{general-subordination}).
We point out that this approach also works for Theorem \ref{clt-main-cauchy-a2}.

\begin{proof} Fix $\nu>0$, and consider the inverse Gaussian convolution semigroup $(\mu_t)_{t\ge0}$ on $(\mathbb R,+)$
as in (\ref{inverse-gaussian}). For $s\ge0$ and $\beta>0$ let $\rho_{s,\beta}\in M^1\big(C_N^B\big)$
 be the distributions of a Bessel process $(X_{s,\beta})_{s\ge0}$ of type $B_N$ starting in $0$
 with parameter $k=(\nu \beta,\beta)$ as in Theorem \ref{clt-main-b1}. Hence, by Theorem~\ref{clt-main-b1},
 $X_{1,\beta} - \sqrt{\beta } r$ tends in distribution to ${\mathcal N}(0,\Sigma)$ with the covariance matrix
 $\Sigma$ from Theorem~\ref{clt-main-b1}.
 Therefore, in terms of the classical convolution $*$ of measures on~$(\mathbb R^n,+)$,
\[ \rho_{1,\beta}*\delta_{-\sqrt{\beta} r}\longrightarrow {\mathcal N}(0,\Sigma), \qquad \beta\to\infty,\]
weakly. Hence, using the classical Fourier transform $\hat\mu(w):=\int_{\mathbb R^N} {\rm e}^{-{\rm i}\langle w,x\rangle}\,{\rm d}\mu(x)$
of measures $\mu$ on $\mathbb R^N$ and Levy's continuity theorem, we get
\begin{equation}\label{fourier-limit}
{\rm e}^{{\rm i}\sqrt{\beta}\langle w,r\rangle} \hat \rho_{1,\beta}(w)\longrightarrow {\rm e}^{-w^{\rm T}\Sigma w/2},\qquad \beta\to\infty,
\end{equation}
locally uniformly for $w\in\mathbb R^N$. Moreover, by the scaling properties of the $ \rho_{s,\beta}$ we have
$ \hat\rho_{s,\beta}(w)=\hat\rho_{1,\beta}\big(\sqrt{s}w\big)$ for $s\ge0$ and $w\in\mathbb R^N$.

We now consider the Cauchy--Bessel distributions $\tau_{\nu,\beta}$ which are related to the $\rho_{s,\beta}$ via the subordination
(\ref{general-subordination}) by
\[\tau_{\nu,\beta}=\int_0^\infty \rho_{s,\beta}\,{\rm d}\mu_{\sqrt 2}(s)\]
in the sense of concatenation of a Markov kernel with a measure. Using the definition (\ref{inverse-gaussian}) of~$\mu_{\sqrt 2}$,
we obtain
\begin{align*}
 \hat \tau_{\nu,\beta}(w)&=\frac{1}{\sqrt{2\pi}}\int_0^\infty\hat\rho_{s,\beta}(w) s^{-3/2}\exp(-1/(2s))\,{\rm d}s \\
 &=\frac{1}{\sqrt{2\pi}}\int_0^\infty\hat\rho_{1,\beta}\big(\sqrt{s}w\big) s^{-3/2}\exp(-1/(2s))\,{\rm d}s.
 \end{align*}
Now consider the linear mappings (\ref{lin-mp-b}) which transform the given random variables $X_\beta$ with the distributions
$\tau_{\nu,\beta}$ into the rescaled random variables $\tilde X_\beta:=\phi_\beta( X_\beta)$.
As the $\phi_\beta$ are symmetric linear operators, we conclude that the
Fourier transforms $\hat \tau_{\nu,\beta,\phi}$ of the distributions $\tau_{\nu,\beta,\phi}$ of the~$\tilde X_\beta$ satisfy
\[\hat \tau_{\nu,\beta,\phi}(w)= \hat \tau_{\nu,\beta}(\phi_\beta(w)).\]
Using~(\ref{fourier-limit}) and dominated convergence, we hence obtain that for $w\in \mathbb R^N$,
\begin{align*}
 \lim_{\beta\to\infty} \hat \tau_{\nu,\beta,\phi}(w)&=\lim_{\beta\to\infty} \hat \tau_{\nu,\beta}(\phi_\beta(w)) \\
 &= \frac{1}{\sqrt{2\pi}} \int_0^\infty
\lim_{\beta\to\infty} \hat\rho_{1,\beta}\big(\sqrt{s}\phi_\beta(w)\big) s^{-3/2}\exp(-1/(2s))\,{\rm d}s \\
&= \frac{1}{\sqrt{2\pi}}
\int_0^\infty \lim_{\beta\to\infty}\big( {\rm e}^{-s \phi_\beta(w)^{\rm T}\Sigma \phi_\beta(w)/2} {\rm e}^{-{\rm i}\sqrt{s\beta}\langle \phi_\beta(w),r\rangle}\big)
 s^{-3/2}\exp(-1/(2s))\,{\rm d}s.
\end{align*}
We next observe from the definition of $\phi_\beta$ that
\[ \sqrt{\beta}\langle \phi_\beta(w),r\rangle=\langle p_r(w),r\rangle= \langle w,r\rangle\]
and
\[\lim_{\beta\to\infty} \phi_\beta(w)^{\rm T}\Sigma \phi_\beta(w)= p_{r^{\perp}}(w)^{\rm T}\Sigma p_{r^{\perp}}(w).\]
Therefore,
\[ \lim_{\beta\to\infty} \hat \tau_{\nu,\beta,\phi}(w)=\frac{1}{\sqrt{2\pi}}
\int_0^\infty {\rm e}^{-s p_{r^{\perp}}(w)^{\rm T}\Sigma p_{r^{\perp}}(w)/2} {\rm e}^{-{\rm i}\sqrt{s} \langle w,r\rangle}
 s^{-3/2}\exp(-1/(2s))\,{\rm d}s.\]
Clearly, the r.h.s.\ is just the Fourier transform of the probability measure{\samepage
\[\mu:=\frac{1}{\sqrt{2\pi}}\int_0^\infty {\mathcal N}(0,s A\Sigma A)* \delta_{\sqrt s r} s^{-3/2}\exp(-1/(2s))\,{\rm d}s
=\int_0^\infty {\mathcal N}\big(\sqrt s r,s A\Sigma A\big)\, {\rm d}\mu_{\sqrt 2}(s)\]
from (\ref{hkb-final2-theorem-intrep}). Hence, by Levy's continuity theorem, we have weak convergence to this $\mu$.}

We finally check that $\mu$ has the density (\ref{hkb-final2-theorem}).
Clearly we may restrict our attention to the interior of the half space $B_N$.
Moreover, by standard arguments from measure theory
it suffices to check this by comparing $\mu$ with the measure with density (\ref{hkb-final2-theorem})
for sets of the form $T([c_1,d_1]\times R)\subset B_N$ for $0\le c_1\le d_1$,
a Borel set $R\subset\mathbb R^{n-1}$, and $T$ the map belonging to the change of coordinates
from the given standard coordinates $e_1,\dots,e_N$
into the orthogonal coordinates belonging to the normalized eigenvectors
$v_1,\dots, v_N$ of $\Sigma^{-1}$
associated with the eigenvalues
$2,4,\dots,N$.
We recapitulate that by \cite[Theorem 4.3]{AV2} and by~(\ref{equality-F-B0}),
$v_1=r/\sqrt{N(N+\nu-1)}$ holds.
With these notations and the substitution $t=\sqrt s \|r\|$
we obtain for the probability measure $\mu$ from~(\ref{hkb-final2-theorem-intrep}) that
\begin{align}
 &\mu(T ([c_1,d_1]\times R)) \notag\\
 &\qquad{}= \frac{1}{\sqrt{2\pi}}\int_0^\infty
 {\mathcal N}\left(\sqrt s \|r\| e_1,s \operatorname{diag}\left(0,\frac{1}{4},\frac{1}{6},\dots,\frac{1}{2N}\right)\right)([c_1,d_1]\times R)
 \frac{ \exp(-1/(2s))}{ s^{3/2}}\,{\rm d}s\notag\\
&\qquad{} =\frac{2\|r\|}{\sqrt{2\pi}}\int_0^\infty {\mathcal N}\left(t e_1,\frac{t^2}{\|r\|^2} \operatorname{diag}\left(0,\frac{1}{4},\frac{1}{6},\dots,\frac{1}{2N}\right)\right)([c_1,d_1]\times R)
 \frac{ {\rm e}^{-\|r\|^2/(2t^2)}}{ t^{2}}\,{\rm d}t\notag\\
& \qquad{}=\frac{2\|r\|}{\sqrt{2\pi}} \int_{c_1}^{d_1}
 {\mathcal N}_{N-1}\left(0,\frac{t^2}{\|r\|^2} \operatorname{diag}\left(\frac{1}{4},\frac{1}{6},\dots,\frac{1}{2N}\right)\right)( R)
 \frac{ {\rm e}^{-\|r\|^2/(2t^2)}}{ t^{2}}\,{\rm d}t,\label{mu-ev}
\end{align}
where ${\mathcal N}_{N-1}$ is an $(N-1)$-dimensional normal distribution. On the other hand, with the same change of coordinates,
(\ref{hkb-final2-theorem}) and
(\ref{equality-F-B0}) lead to
\begin{align}
 &\int_{ T([c_1,d_1]\times R)}f(y)\,{\rm d}y\notag\\
&\qquad{} = D(N) \int_{c_1}^{d_1}
 \Biggl(\int_{R}
\exp\left(-\frac{\|{r }\|^2}{2t^2}(t,y)^{\rm T} \operatorname{diag}(2,4,6,\dots,2N)(t,y)\right) {\rm d}y \Biggr)\notag\\
&\qquad\quad{}\times \exp\left(\frac{-\|{r }\|^2 }{2t^2}\right) \frac{1}{ t^{N+1}}\, {\rm d}t\notag\\
&\qquad{} D(N) {\rm e}^{- \|r\|^2} \int_{c_1}^{d_1}\Biggl(\int_{ R}
\exp\left({-}\frac{\|{r }\|^2}{2t^2}y^{\rm T} \operatorname{diag}(4,6,\dots,2N)y\right){\rm d}y \Biggr)\notag\\
&\qquad\quad{}\times \exp\left(\frac{- \|r\|^2}{2t^2}\right) \frac{1}{ t^{N+1}}\,{\rm d}t.\label{norming-final-b}
\end{align}
Using the definition of $D(N)$ and the constants of multivariate normal distributions, we see that the expressions in the end of
(\ref{mu-ev}) and (\ref{norming-final-b}) are equal. This completes the proof.
 \end{proof}

\begin{Remark}\label{b-final-remarks}\quad
\begin{enumerate}\itemsep=0pt
\item[(1)] Clearly, the central limits Theorem~\ref{clt-main-cauchy-a2} can be also proved in the same way as Theorem~\ref{clt-main-cauchy-b}. Moreover
the limit measure with density~(\ref{hka-final2-theorem}) there can be also
expressed in a form which corresponds to~(\ref{hkb-final2-theorem-intrep}).

On the other hand, the methods of the proof of
Theorem \ref{clt-main-cauchy-a2} can be also applied in the situations of Theorems~\ref{clt-main-cauchy-a2} and~\ref{clt-main-cauchy-b} in order
to derive corresponding limits for distributions of the form
\begin{equation*}
c(k,t)
 \left(\frac{4}{ t^2+2\|y\|^2}\right)^{r_k} w_k(y)
\end{equation*}
with more general exponents $r_k$ than in (\ref{density-general-cauchy}) and suitable norming constants $c(k,t) >0$.
\item[(2)] The asymptotic Theorem \ref{clt-main-cauchy-a} in the Hermite case can be also transfered to the Laguerre case.
 We skip the details.
\item[(3)] The assertion of Theorem \ref{clt-main-cauchy-b} remains valid for all Cauchy--Bessel distributions of type $B_N$
as defined in (\ref{general-subordination}) via subordination
for all fixed starting points $x\in C_N^B$ and not only for $x=0$.

This can be seen as follows. Lemma~5 of~\cite{AKM2} implies that for all $x,y\in C_N^B$ and $\nu>0$, the corresponding Bessel functions satisfy
\begin{equation}\label{limit-bessel-function-b}
 \lim_{\beta\to\infty} J_{(\nu \beta, \beta)}^B\big(\sqrt \beta x, y\big)=
 \exp\left( \frac{\|x\|^2\|y\|^2}{4N(\nu+N-1)}\right).
 \end{equation}
This implies that Theorem \ref{clt-main-b1} is available also for arbitrary starting points $x\in C_N^B$; see \cite[Theorem 3.3]{V}.
This shows that the proof of Theorem \ref{clt-main-cauchy-b} also works for arbitrary starting points $x\in C_N^B$.
\item[(4)] The preceding result can be also stated for the root systems $A_{N-1}$ and arbitrary starting points $x\in C_N^A$.
 However, the details of the proof and of the result are slightly more complicated, as the root system is not longer reduced, and
 as the center-of-gravity-part of the limit has a slightly different behavior. In fact, the analogue of~(\ref{limit-bessel-function-b})
 for the Bessel functions of type A is more complicated; see of \cite[Corollary~8]{AM}
 as well as \cite[Lemma~2.4 and Theorem~2.5]{AV2}.
 This limit for the Bessel functions implies that the limit distribution of the CLT for
 Bessel processes in \cite[Theorem~2.3]{AV2} contains an additional drift in the
 center-of-gravity-direction. Having this in mind,
 one can also restate central limits Theorem~\ref{clt-main-cauchy-a2} in this way for arbitrary starting points $x\in C_N^A$
by taking this drift into account.
\item[(5)] In \cite{VW}, freezing limits are studied for Bessel processes with parameter $k\to\infty$
 where the starting points have the form $\sqrt k x$ with points $x$ in the interior of the Weyl chamber.
 We do not know whether the CLTs there can be transfered to Cauchy--Bessel processes.
\item[(6)] We expect that the methods of the proof of Theorem~\ref{clt-main-cauchy-b} can be used to study freezing limits
for further classes of distributions which appear form the Bessel processes by different subordinations like general analogues of stable distributions.
\item[(7)] In the singular case $\nu=0$ there exists an analogue of Theorem~\ref{clt-main-cauchy-b} where the details are slightly different.
 We discuss this singular case in the next section as a consequence of the corresponding results for the root systems $D_N$.
 \end{enumerate}
 \end{Remark}

\section[Freezing limits for the root system D\_N and an extremal B\_N-case]{Freezing limits for the root system $\boldsymbol{D_N}$\\ and an extremal $\boldsymbol{B_N}$-case}\label{section5}

We here briefly study Bessel processes and related Cauchy--Bessel distributions for the root system $D_N$ and an extremal $B_N$-case.
 We recapitulate that the root system $D_N$
is given by
\[D_N=\{\pm e_i\pm e_j\colon 1\le i<j\le N\}\]
with the Weyl chamber
\[ C_N^D=\big\{x\in\mathbb R^N\colon x_1\ge \dots\ge x_{N-1}\ge |x_N|\big\},\]
which may be seen as a doubling of $C_N^B$ w.r.t.\ the last coordinate. We have a multiplicity $k\in{} ]0,\infty[$.
The generator of the transition semigroup of the Bessel process $(X_{t,k})_{t\ge0}$ of type~D~is
\begin{gather*}
Lf:= \frac{1}{2} \Delta f +
 k \sum_{i=1}^N \sum_{j\ne i} \left( \frac{1}{x_i-x_j}+\frac{1}{x_i+x_j} \right)
 \frac{\partial}{\partial x_i}f. \end{gather*}
 The transition probabilities are
\begin{equation*}
K_{t,k}(x,A)=c_k^D \int_A \frac{1}{t^{\gamma_D+N/2}} {\rm e}^{-(\|x\|^2+\|y\|^2)/(2t)} J_k^D\left(\frac{x}{\sqrt{t}}, \frac{y}{\sqrt{t}}\right)
 w_k^D(y)\, {\rm d}y
\end{equation*}
with
\begin{equation}\label{data-D}
 w_k^D(x):= \prod_{i<j}\big(x_i^2-x_j^2\big)^{2k}, \qquad \gamma_D:= kN(N-1)
\end{equation}
and
\begin{equation}\label{norming-constant-D}
 c_k^D=
 \frac{N!}{2^{N(N-1)k-N/2+1}} \prod_{j=1}^{N}\frac{\Gamma(1+k)}{\Gamma(1+jk)\Gamma(\frac{1}{2}+(j-1)k)};
\end{equation}
see Demni \cite{Dem} and \cite{AV1, V} for the details.

We next recapitulate some fact on Laguerre polynomials.
Using the representation
\[L_N^{(\alpha)}(x):=\sum_{k=0}^N { N+\alpha\choose N-k}\frac{(-x)^k}{k!}\]
(see \cite[equation~(5.1.6)]{S}), we can form the polynomial $L_N^{(-1)}$ of order $N\ge1$
where,
 by \cite[equation~(5.2.1)]{S},
\begin{equation}\label{laguerre-1}
L_N^{(-1)}(x)=-\frac{x}{N}L_{N-1}^{(1)}(x).
\end{equation}
Using the $N-1$ ordered zeros $z_1^{(1)}>\dots>z_{N-1}^{(1)}>0$ of
 $L_{N-1}^{(1)}$, we define the vector $r=(r_1,\dots,r_N)\in [0,\infty[^N$ with
 \begin{equation}\label{def-r-d}
 2 \big(z_1^{(1)},\dots, z_{N-1}^{(1)},0\big)= \big(r_1^2,\dots,r_N^2\big)
 \end{equation}
 similar to Section~\ref{section4}. Notice that $r$ is in the interior of $C_N^D$, and that
 (\ref{equality-F-B0}) and (\ref{laguerre-1}) imply
 \begin{equation*}
 \|r\|^2=N(N-1).
\end{equation*}
 Most parts of the
 following CLT for the Bessel processes $(X_{t,k})_{t\ge0}$ of type $D_N$ on $C_N^D$ with multiplicity $k>0$ with start
 in $0$ were proved in~\cite{V}:

\begin{Theorem}\label{clt-main-D}
For each $t>0$, the random variables
$\frac{X_{t,k}}{\sqrt t} - \sqrt{k } r$
converge for $k\to\infty$ to the centered $N$-dimensional distribution ${\mathcal N}(0, \Sigma_D)$
with the regular covariance matrix $\Sigma_D$ with $\Sigma_D^{-1}=(s_{i,j})_{i,j=1,\dots,N}$ with
\begin{equation}\label{covariance-matrix-D}
s_{i,j}:= \begin{cases} \displaystyle 1+ 2\sum_{l\ne i} (r_i-r_l)^{-2}+2\sum_{l\ne i} (r_i+r_l)^{-2} &
 \text{for}\ i=j, \\
 2(r_i+r_j)^{-2} -2(r_i-r_j)^{-2} & \text{for}\ i\ne j. \end{cases}
\end{equation}
The entries $s_{i,j}$ satisfy $s_{i,N}=s_{N,i}=0$ for $i=1,\dots,N-1$ and $s_{N,N}=N$.
 The block $(s_{i,j})_{i,j=1,\dots,N-1}$
 is the inverse covariance matrix in Theorem~{\rm \ref{clt-main-cauchy-b}} for the dimension $N-1$ with $\nu=2$.
\end{Theorem}

\begin{proof} By the proof of Theorem 5.2 in \cite{V}, the densities of the
 $\frac{X_{t,k}}{\sqrt t} - \sqrt{k } r$ may be written as
\begin{align*}
f_k^D(y)
={}& c_k^D
 \exp\Bigg(2k \sum_{i<j}\ln\bigg(1+ \frac{y_i-y_j}{\sqrt{k}(r_i-r_j)}\bigg)+
2k \sum_{i<j}\ln\bigg(1+ \frac{y_i+y_j}{\sqrt{k}(r_i+r_j)}\bigg)\Bigg) \\
& {}\times {\rm e}^{-\|y\|^2/2} {\rm e}^{-k\|r\|^2/2} {\rm e}^{-\sqrt{k} \langle y,r\rangle} \exp\Bigg(2k \sum_{i<j}\bigl(\ln\big(\sqrt{k}( r_i-r_j)\big) +\ln\big(\sqrt{k}( r_i+r_j)\big)\bigr)\Bigg)
\end{align*}
on the shifted cone $C_N^D-\sqrt{k} r$, with $f_k^D(y)=0$ elsewhere on $\mathbb R^N$.
We write this as
\[f_k^D(y)= \tilde c_k^D h_k(y)\]
with
\begin{align*}
h_k(y):={}& \exp\Bigg({-}\|y\|^2/2 -\sqrt{k} \langle y,r\rangle +
2k \sum_{i<j}\bigg(\ln\bigg(1+ \frac{y_i-y_j}{\sqrt{k}(r_i-r_j)}\bigg)\\
&{}+\ln\bigg(1+ \frac{y_i+y_j}{\sqrt{k}(r_i+r_j)}\bigg)\bigg)\Bigg)
\end{align*}
and
\begin{equation*}
 \tilde c_k^D :=c_k^D {\rm e}^{-k\|r\|^2/2} \exp\Bigg(2k \sum_{i<j}\big(\ln\big(\sqrt{k}( r_i-r_j)\big) +\ln\big(\sqrt{k}( r_i+r_j)\big)\big)\Bigg),
\end{equation*}
 where, by \cite[equation~(5.7)]{V},
\begin{equation*}
\lim_{k\to\infty}h_k(y)= \exp\Bigg({-}\frac{\|y\|^2}{2}-\sum_{i<j}\frac{(y_i-y_j)^2}{(r_i-r_j)^2}
-\sum_{i<j}\frac{(y_i+y_j)^2}{(r_i+r_j)^2}\Bigg).
\end{equation*}
This implies by the arguments in the proofs of Theorem~5.2 in~\cite{V}
(more precisely, by the arguments in the proofs of Theorems~2.2 and~3.3 there) that
the probability measures with the densities $f_k^D$ tend weakly to ${\mathcal N}(0, \Sigma_D)$ with $\Sigma_D^{-1}$ as in the theorem above.
Moreover, except for the statement $s_{N,N}=N$, all additional facts about the entries of $\Sigma_D^{-1}$ in the theorem are clear by~(\ref{laguerre-1}).

In order to prove $s_{N,N}=N$, we
use (\ref{laguerre-1}) and (\ref{equality-F-B}) for $\nu=2$ (i.e., $\alpha=1$) and $N-1$ (instead of $N$), and we observe that
in our situation $r_N=0$ holds. These facts lead readily
to
\begin{equation*}
2\ln\Bigg(\prod_{i<j}\big( r_i^2 -r_j^2\big)\Bigg)= N(N-1)(-1/2+\ln 2)+ \sum_{j=1}^N j\ln j + \sum_{j=1}^{N-1} j\ln j.
\end{equation*}
This, (\ref{norming-constant-D}), and Stirling's formula applied to the Gamma functions in~(\ref{norming-constant-D}) now imply that
\begin{equation*}
 \lim_{k\to\infty} \tilde c_k^D =\frac{2^{(N-1)/2} \sqrt{N!}}{(2\pi)^{N/2}}.
\end{equation*}
If we compare this with the normalization constants of ${\mathcal N}(0, \Sigma_D)$ and use
\[
\det\big(\Sigma_D^{-1}\big)=\det( (s_{i,j})_{i,j=1,\dots,N-1}) s_{N,N}= 2^{N-1} (N-1)! s_{N,N},\]
we obtain $s_{N,N}=N$ as claimed.
\end{proof}

\begin{Remark} If we combine $s_{N,N}=N$ with the $(N,N)$-entry in (\ref{covariance-matrix-D}), we obtain that the zeros
 $z_{1}^{(1)}>\dots>z_{N-1}^{(1)}>0$ of $L_{N-1}^{(1)}$ satisfy
 $\sum_{l=1}^{N-1} \frac{1}{z_{l}^{(1)}}=\frac{N-1}{2}$.
 It was pointed out by one of the referees that such sums over
 the inverses of the zeros can be computed easily for all classical orthogonal polynomials.
 For this use the elementary symmetric polynomials $e_0,\dots, e_N$ in $N$ variables, and write such a polynomial $P_N$ of order $N$ with
 zeros $z_1,\dots, z_N$ as
 \begin{equation}\label{rep}
 P_N(z)=c_N\prod_{j=1}^N(z-z_j)=c_N\sum_{j=0}^N (-1)^{N-j}e_{N-j}(z_1,\dots, z_N) z^j.\end{equation}
 As
\[\sum_{l=1}^{N} \frac{1}{z_{l}}=\frac{e_{N-1}(z_1,\dots, z_N)}{e_{N}(z_1,\dots, z_N)},\]
 we can derive this sum from (\ref{rep}) and the well-known formulas for the coefficients of the classical orthognal polynomials
 in \cite{S}. For instance, equation~(5.1.6) of \cite{S} yields for $L_N^{(\alpha)}$ with the zeros
 $z_1^{(\alpha)},\dots, z_N^{(\alpha)}$ that
 \begin{equation*}
 \sum_{l=1}^{N} \frac{1}{z_{l}^{(\alpha)}}=\frac{N}{\alpha+1}.
 \end{equation*}
 This in particular leads to an alternative proof of the statement $s_{N,N}=N$ in the preceding theorem.
 \end{Remark}

We next turn to Cauchy--Bessel distributions of type $D_N$ which are constructed from the associated Bessel processes
via subordination. More precisely, we use the inverse Gaussian distribution $\mu_t$ with density
(\ref{inverse-gaussian}) for $t=\sqrt 2$ as in the preceding sections, and obtain from the densities
(\ref{density-general-0}) together with (\ref{density-general-cauchy}), (\ref{data-D}), and (\ref{norming-constant-D}) that the
associated Cauchy--Bessel ensembles have the densities
\begin{equation}\label{density-general-cauchy-d}
f_{k,D}(y):=C_D(k,N) \frac{1}{( 1+\|y\|^2)^{k N(N -1)+(N+1)/2}} \prod_{i,j\colon i<j}\big(y_i^2-y_j^2\big)^{2k}
\end{equation}
on the Weyl chambers $C_N^D$ with the norming constants
\begin{equation*}
C_D(k,N) =
\frac{ 2^{N-1} N! \Gamma(k N(N-1) + (N+1)/2)}{\sqrt\pi }
 \prod_{j=1}^{N}\frac{\Gamma(1+k)}{\Gamma(1+jk)\Gamma(\frac{1}{2}+k(j-1))}.
\end{equation*}
The Fourier-analytic proof of Theorem \ref{clt-main-cauchy-b} leads to the following CLT where, similar to Section~\ref{section4}, we use the normalization
 mappings
\begin{equation}\label{lin-mp-d}
 \phi_k\colon \ \mathbb R^N\to \mathbb R^N, \qquad \phi_k(y):=
 \frac{1}{\sqrt k} p_{r}(y) +p_{{r}^\perp}(y)= y+\left(\frac{1}{\sqrt k }-1\right)p_{{r}}(y).
 \end{equation}

\begin{Theorem}\label{clt-main-cauchy-d}
 For $k>0$ and $N\ge2$ let $X_k$ be $C_N^D$-valued random variables with the Lebesgue densities \eqref{density-general-cauchy-d}.
Then the rescaled random variables $\tilde X_k:= \phi_k(X_k)$
 converge in distribution for $k\to\infty$ to some $\mu\in M^1\big(\mathbb R^N\big)$ with
 the half space
 $B_N:=\big\{y\in\mathbb R^N\colon \langle y, r\rangle \ge 0\big\} $ as support. $\mu$ is given by
 \begin{equation*}
 \mu:=\frac{1}{\sqrt{2\pi}}\int_0^\infty {\mathcal N}\big(\sqrt s r,s A\Sigma_D A\big) s^{-3/2} \exp(-1/(2s))\,{\rm d}s,
 \end{equation*}
where $ A$ is the matrix of the orthogonal projection $p_{{r}^\perp}$. The measure
 $\mu$ has the Lebesgue density
\begin{gather*}
 f(y):= D(N)\exp\left(-\frac{N(N-1)}{2\| p_{{r }}(y)\|^2} y^{\rm T} \Sigma_D^{-1} y \right)
 \exp\left(\frac{- N(N-1)}{ \|p_{{r }}(y)\|^2}\right) \frac{1}{ \| p_{{r}}(y)\|^{N+1}}
 \end{gather*}
for $y$ in the interior of $B_N$ with
\[D(N):=\frac{ \sqrt{N!} (N(N-1))^{N/2}{\rm e}^{N(N-1)} }{\pi^{N/2} }.\]
\end{Theorem}

The central limit Theorems~\ref{clt-main-D} and \ref{clt-main-cauchy-d} for Bessel and Cauchy--Bessel processes of type~D lead immediately to
CLTs for the Bessel and Cauchy--Bessel processes of type B with the multiplicities $(k_1,k_2):=(0,\beta)$ for $\beta\to\infty$, i.e.,
the case $\nu=0$ in Section~\ref{section4}.

For this we recapitulate the following fact from~\cite{AV1}.
If $\big(X_{t,k}^D\big)_{t\ge0}$ is a Bessel process of type D with multiplicity $k\ge0$ on the chamber $C_N^D$ starting in $0$, then the process
$\big(X_{t,k}^B\big)_{t\ge0}$ with
\[X_{t,k}^{B,i}:= X_{t,k}^{D,i}, \quad i=1,\dots,N-1, \qquad X_{t,k}^{B,N}:= \big|X_{t,k}^{D,N} \big|\]
is a Bessel process of type B with $(k_1,k_2):=(0,k)$. This follows easily from a comparison of the
corresponding generators.
The central limit Theorem~\ref{clt-main-D} for $\big(X_{t,k}^D\big)_{t\ge0}$ thus leads
to the following central limit Theorem~\ref{clt-main-b-one-sided} for
Bessel processes of type B with the multiplicities $(0,k)$ for $k\to\infty$ with
one-sided normal distribution as limit; see \cite[Corollary~5.3]{V}.
By \cite[Theorem~6.2]{V}, this CLT also holds for the multiplicities $(k_1,k_2)$ form any fixed $k_1\ge0$ and $k_2\to\infty$.

To state the result we denote the image of a $N$-dimensional normal distribution ${\mathcal N}(0,\Sigma)$ with covariance matrix $\Sigma$
under the map
\[\mathbb R^N \longrightarrow H_N:=\big\{x\in\mathbb R^N\colon x_N\ge0\big\}, \qquad (x_1,\dots,x_N)\mapsto (x_1,\dots,x_{N-1}, |x_N|)\]
by $|{\mathcal N}(0,\Sigma)|$, i.e., the support of $|{\mathcal N}(0,\Sigma)|$ is contained in the half space~$H_N$.

\begin{Theorem}\label{clt-main-b-one-sided}
 Consider the Bessel processes $(X_{t,(k_1,k_2)})_{t\ge0}$ of type $B_N$ on $C_N^B$ with multiplicities $(k_1,k_2)$ with start in $0$ and
 $k_1\ge0$.
Then, for the vector $r$ from \eqref{def-r-d} on the boundary of~$C_N^B$,
\[\frac{X_{t,(k_1,k_2)}}{\sqrt t} - \sqrt{k_2 } r\]
converges for $k_2\to\infty$ in distribution to $|{\mathcal N}(0, \Sigma_D)|$
with $\Sigma_D$ as in Theorem~{\rm \ref{clt-main-D}}.
\end{Theorem}

This one-sided CLT leads to the following corresponding result for Cauchy--Bessel distributions.

\begin{Corollary}\label{clt-main-b-one-sided-cauchy}
 For $k_1,k_2\ge 0$ and integer $ N\ge2$ let $X_{k_1,k_2}$ be $C_N^B$-valued random variables with the Cauchy--Bessel densities
\begin{equation*}
f_{k_1,k_2}(y):=C_B(k_1,k_2,N)
\frac{1}{\big( 1+\|y\|^2\big)^{k_2 N(N -1)+k_1N+(N+1)/2}} \prod_{i,j\colon i<j}\big(y_i^2-y_j^2\big)^{2k_2}\prod_{i=1}^N y_i^{2k_1}
\end{equation*}
with the norming constants as in \eqref{density-general-cauchy-b-norming}.
Then the rescaled random variables $\tilde X_{k_1,k_2}:= \phi_k(X_{k_1,k_2})$ with $\phi_k$ as in \eqref{lin-mp-d}
 converge in distribution for $k_2\to\infty$ to some $\mu\in M^1\big(\mathbb R^N\big)$ with
 the quarter space
\[B_{N,0}:=\big\{y\in\mathbb R^N\colon y_N\ge 0, \langle y, r\rangle \ge 0\big\}\]
 as support. $\mu$ is given by
 \begin{equation*}
 \mu:=\frac{1}{\sqrt{2\pi}}\int_0^\infty \big|{\mathcal N}\big(\sqrt s r,s A\Sigma_D A\big)\big| s^{-3/2}\exp(-1/(2s))\,{\rm d}s,
 \end{equation*}
where $ A$ is the matrix of the orthogonal projection $p_{{r}^\perp}$. The measure
 $\mu$ has the Lebesgue density
\begin{gather*}
 f(y):= 2D(N)\exp\left(-\frac{N(N-1)}{2\| p_{{r }}(y)\|^2} y^{\rm T} \Sigma_D^{-1} y \right)
 \exp\left(\frac{- N(N-1)}{ \|p_{{r }}(y)\|^2}\right) \frac{1}{ \| p_{{r}}(y)\|^{N+1}}
 \end{gather*}
for $y$ in the interior of $B_{N,0}$ with $D(N)$ as in Theorem~{\rm \ref{clt-main-cauchy-d}}.
\end{Corollary}

\subsection*{Acknowledgements}
The author would like to thank the anonymous referees for their numerous comments, which improved the paper considerably.



\pdfbookmark[1]{References}{ref}
\LastPageEnding

\end{document}